\documentclass[a4paper,12pt]{article}
\usepackage{amsmath,amsfonts,amssymb,amsthm}

\newcommand{\Rset}{\mathbb{R}}
\newcommand{\Nset}{\mathbb{N}}

\newcommand{\Cset}{\mathbb{C}}

\newcommand{\PP}{\mathbb{P}}
\newcommand{\EE}{\mathbb{E}}
\newcommand{\DS}{\displaystyle}

\newcommand{\sostituzioneA}{ }    

\newtheorem{Theorem}{Theorem}[section]
\newtheorem{Definition}[Theorem]{Definition}
\newtheorem{Proposition}[Theorem]{Proposition}
\newtheorem{Lemma}[Theorem]{Lemma}
\newtheorem{Corollary}[Theorem]{Corollary}
\newtheorem{Remark}[Theorem]{Remark}

\newtheorem{Hypothesis}[Theorem]{Hypothesis}

\title{Measure-valued equations for Kolmogorov operators with unbounded   coefficients}
\author{Luigi Manca\footnote{
Dipartimento di Matematica Pura e Applicata,
Universit\`a di Padova,
Via Trieste 63, 
35121 Padova, Italy, \emph{E-mail: manca@math.unipd.it}}}
\date{July 21, 2007}
\begin{document}

\maketitle

\begin{abstract} 
Given a real and separable Hilbert space $H$ we consider 
the measure-valued equation 
\begin{equation*}
\int_H\varphi(x)\mu_t(dx)- \int_H\varphi(x)\mu(dx)=
  \int_0^t\left( \int_HK_0\varphi(x)\mu_s(dx)   \right)ds, 
\end{equation*}
where $K_0$ is the Kolmogorov differential operator
\[
  K_0\varphi(x)=\frac12\textrm{Trace}\big[BB^*D^2\varphi(x)\big]+\langle x,A^*D\varphi(x)\rangle+\langle D\varphi(x),F(x)\rangle,
\]
$x\in H$,  $\varphi:H\to \Rset$ is a suitable smooth function, $A:D(A)\subset H\to H $ is linear,     $F:H\to H$ is a globally Lipschitz function and $B:H\to H$ is linear and continuous.
In order prove existence and uniqueness of a solution for the above equation, we show that $K_0$ is a core, in a suitable way, of the infinitesimal generator associated  to the solution of a certain stochastic differential equation in $H$.

We also extend the above results 
to a reaction-diffusion operator with polinomial nonlinearities. 
\end{abstract}
\section{Introduction}
Let $H$ be a separable real Hilbert space (with norm $|\cdot|$ and inner product $\langle\cdot,\cdot\rangle $), and let $\mathcal B(H)$ be its Borel $\sigma$-algebra.
$\mathcal L(H)$ denotes the usual Banach space of all   linear and continuous operators  in $H$, endowed with the supremum norm $\|\cdot\|_{\mathcal L(H)}$.
 We consider the 
stochastic differential equation in $H$ 
\begin{equation} \label{e.L.3.1}
\left\{\begin{array}{lll}
dX(t)&=& \big(AX(t)+F(X(t))\big)dt+B dW(t),\quad t\geq 0\\
\\
X(0)&=&x\in H,
\end{array}\right.
\end{equation}
where 
\begin{Hypothesis}  \label{h.L.3.1}
\begin{itemize}

\item[(i)] $A\colon D(A)\subset H\to H$ is the infinitesimal
generator of a strongly continuous semigroup $e^{tA}$ of type ${\cal G}(M,\omega)$, i.e. there exist $M\geq 0$ and $\omega\in \Rset$ such that $\|e^{tA}\|_{{\mathcal L}(H)}\le Me^{\omega t}$, $ t\ge 0$;
\item[(ii)] 
$B\in \mathcal L(H)$  and for any $t> 0$ the linear operator $Q_{t}$, defined by  
\begin{equation*}  
Q_{t}x=\int_{0}^{t}e^{sA}BB^*e^{sA^{*}}x\,ds,\;\;x\in H,\;t\ge 0
\end{equation*}
has finite trace;
\item[(iii)] $F:H\to H$ is a Lipschitz continuous map.
We set \\$\DS \kappa=\sup_{\stackrel{x,y\in H}{ x\not= y}}\frac{|F(x)-F(y)|}{|x-y|}$;
\item[(iv)] $(W(t))_{t\geq0}$ is a cylindrical Wiener process, defined on a stochastic basis $(\Omega,\mathcal F, (\mathcal F_t)_{t\geq0},\PP)$ and with values in $H$.
\end{itemize}
\end{Hypothesis}
It is well known that under hypothesis \eqref{h.L.3.1} problem \eqref{e.L.3.1} has a  unique {\em mild} solution  $(X(t,x))_{t\geq0,x\in H}$ (see, for instance, \cite{DPZ92}), that is for any $x\in H$ the process $(X(t,x))_{t\geq0 }$ is adapted to the filtration $(\mathcal F_t)_{t\geq0} $, it is continuous in mean square
and it fulfils    the integral equation
\begin{equation} \label{e.L.4.3}
 X(t,x)= e^{tA}x+\int_0^te^{(t-s)A}BdW(s)+\int_0^te^{(t-s)A}F(X(s,x))ds
\end{equation}
for any $t\geq0$.
Moreover, a straightforward computation shows that for any $T>0$ there exists $c>0$ such that 
\begin{equation} \label{e.L.3}
 \sup_{t\in[0,T]} |X(t,x)-X(t,y)|\leq 
    c|x-y|,\quad \forall x,\, y\in H, 
\end{equation}
and  
\begin{equation} \label{e.L.10a}
   \sup_{t\in[0,T]} \EE\bigl[|X(t,x)|   \bigr]\leq 
    c(1+ |x| ), \quad   x\in H,
\end{equation}
where the expectation is taken with respect to $\PP$.
Now denote by $C_b(H)$ the Banach space of all uniformly continuous and bounded functions $\varphi:H\to\Rset$, endowed with the supremum norm $\|\varphi\|_0=\sup_{x\in H}|\varphi(x)|$, $\varphi\in C_b(H)$.
Moreover, for any $k>0$, let $C_{b,k}(H)$ be the space  of all functions $\varphi:H\to\Rset$ such that the function $H\to\Rset$, $x\mapsto (1+|x|^k)^{-1}\varphi(x) $ belongs to $C_b(H)$. 
The space $C_{b,k}(H)$ is a Banach space, endowed with the norm $\|\varphi\|_{0,k}= \|(1+|\cdot|^k)^{-1}\varphi\|_0$.
In the following, we shall denote by $\left( C_{b,k}(H) \right)^*$ the topological dual space of $C_{b,k}(H)$.
As we shall see in Proposition \ref{p.L.2.1},  estimates \eqref{e.L.3}, \eqref{e.L.10a} allow us to define the  transition operator     associated to equation \eqref{e.L.4.3} in the space $C_{b,1}(H)$,   by the formula
\begin{equation} \label{e.L.38a}
   {P}_t\varphi(x)=\EE\big[\varphi(X(t,x))\big],\quad \varphi \in C_{b,1}(H),\,  \,t\geq 0,\, x\in H.
\end{equation}
The family of operators $({P}_t)_{t\geq0} $ maps $C_{b,1}(H)$ into $C_{b,1}(H)$ and enjoyes the semigroup property, but it is   not a strongly continuous semigroup 
(cf Proposition \ref{p.L.2.1}). 
However,  we can define the infinitesimal generator of $( {P}_t)_{t\geq0} $ in $C_{b,1}(H) $ in the following way
\begin{equation}  \label{e.L.0}
\begin{cases}
\DS D( {K}\sostituzioneA)=\bigg\{ \varphi \in C_{b,1}(H): \exists g\in C_{b,1}(H), \lim_{t\to 0^+} \frac{ {P}_t\varphi(x)-\varphi(x)}{t}=  \\ 
  \DS  \qquad\qquad\quad\quad=g(x),\,x\in H,\;\sup_{t\in(0,1)}\left\|\frac{ {P}_t\varphi-\varphi}{t}\right\|_{0,1}<\infty \bigg\}\\   
   {}   \\
  \DS  {K}\varphi(x)=\lim_{t\to 0^+} \frac{ {P}_t\varphi(x)-\varphi(x)}{t},\quad \varphi\in D( {K}\sostituzioneA),\,x\in H.
\end{cases}
\end{equation}
Let $\mathcal M(H)$ be the space of all the Borel finite measures on $H$ and for any $k>0$ let $\mathcal M_k(H)$ be the set of all $\mu \in \mathcal M(H) $ such that $\int_H|x|^k|\mu|(dx)<\infty $, where $ |\mu|$ is the total variation of $\mu$. 
The first result of the paper is the following
\begin{Theorem} \label{t.L.1}
Let $( {P}_t)_{t\geq0}$ be the semigroup defined by \eqref{e.L.38a} 
and   let  $( {K},D( {K}\sostituzioneA))$ be its infinitesimal generator in $C_{b,1}(H) $, defined   by  \eqref{e.L.0}.
Then,  the formula 
$$
  \langle \varphi,  {P}_t^*F\rangle_{\mathcal L(C_{b,1}(H),\,(C_{b,1}(H))^*)} = \langle  {P}_t\varphi, F\rangle_{\mathcal L(C_{b,1}(H),\, (C_{b,1}(H))^*)}
$$
defines a semigroup $( {P}_t^*)_{t\geq0}$ of linear and continuous operators  on $(C_{b,1}(H))^*$ that maps $\mathcal M_1(H)$ into $\mathcal M_1(H)$.
Moreover,  for any $\mu\in \mathcal M_1(H)$ there exists a unique family of measures  $\{\mu_t\}_{t\geq0}\subset \mathcal M_1 (H)$
such that
\begin{equation} \label{e.L.5b}
    \int_0^T\left(\int_H|x|^{}|\mu_t|(dx)\right)dt<\infty,\quad \forall T>0 
\end{equation}
and  
\begin{equation}   \label{e.L.8} 
   \int_H\varphi(x)\mu_t(dx)- \int_H\varphi(x)\mu(dx)=
  \int_0^t\left( \int_H {K}\varphi(x)\mu_s(dx)   \right)ds
\end{equation}
 for any $t\geq0$, $\varphi\in D({K}\sostituzioneA)$.
Finally, the solution of \eqref{e.L.8}  is given by $\{ {P}_t^*\mu\}_{t\geq0}$.  
\end{Theorem}

A natural question is to study the above problem with the  
{\em Kolmogorov} differential operator
\begin{equation} \label{e.4}
 K_0\varphi(x)= \frac12\textrm{Tr}\big[BB^*D^2\varphi(x)\big]+\langle x,A^*D\varphi(x)\rangle+\langle D\varphi(x),F(x)\rangle,\,x\in H.
\end{equation}
We stress the fact that the operator ${K}$ is defined in an abstract way, whereas $K_0$ is a {\em concret} differential operator.

In order to study problem \eqref{e.L.8} with    $K_0$ replacing ${K}$, we shall develop the notion of $\pi$-convergence in the spaces  $C_{b,k}(H)$ and the related notion of $\pi$-core.
We recall that the $\pi$-convergence has been introduced in \cite{Priola}, in order to study a class of semigroups that are not strongly continuous.
This notion is one of the key tools   we use to prove our results.

Now let $\mathcal E_A(H)$ be the linear span of the real and imaginary part of the functions
\[
  H\to \Cset,\quad x\mapsto e^{i\langle x,h\rangle}, 
\quad h\in D(A^{*}), 
\]
where $D(A^*)$ is the domain of the adjoint operator of $A$.
We have the following
\begin{Theorem} \label{t.L.2} 
Under Hypothesis \eqref{h.L.3.1}, the operator $(K,D(K\sostituzioneA))$ is an estension of $K_0$, and for any $\varphi\in\mathcal E_A(H)$ we have  $\varphi\in D(K\sostituzioneA)$ and $K\varphi=K_0\varphi$.
Finally, $ \mathcal E_A(H)$ is a $\pi$-core for $(K,D(K\sostituzioneA))$.
\end{Theorem}
As consequence   we have the third main result
\begin{Theorem} \label{t.L.1.4}
For any $\mu\in\mathcal M_1(H)$ there exists an unique family of measures $\{\mu_t\}_{t\geq0}\subset \mathcal M_1(H)$ fulfilling \eqref{e.L.5b} and the measure equation 
\begin{equation}\label{e.L.2.5a}
\int_H\varphi(x)\mu_t(dx)- \int_H\varphi(x)\mu(dx)=
  \int_0^t\left( \int_HK_0\varphi(x)\mu_s(dx)   \right)ds, 
\end{equation}
$t\geq0,\,\varphi \in\mathcal E_A(H)$,
and  the  solution  is given by $\{ {P}_t^*\mu\}_{t\geq0}$.
\end{Theorem}
In \cite{Manca07}  a similar problem when $F:H\to H$ is Lipschitz continuous and bounded has been investigated.
Due to the fact  that the nonlinearity is  bounded, all the result are stated in the space $C_b(H)$.
In our paper  we deal with   unbounded nonlinearities and we need to develop a notion of semigroup and associated infinitesimal generator  in the weighted space $C_{b,1}(H)$.

In section \ref{s.application} we shall extend %
the techniques %
and the results 
of the preceding sections 
to a reaction-diffusion operator with polinomial nonlinearities.

The motivation of this work is to have a better understanding on the relationships between the stochastic differential equation \eqref{e.L.3.1} and the Kolmogorov differential operator $K_0$.
In this direction, the characterization done by Theorems \ref{t.L.2}, \ref{t.RD.2} seems to be new. 

In other papers, the problem of extending a differential operator like \eqref{e.4} to the infinitesimal generator of a diffusion semigroup is studied in the weighted spaces $L^p(H,\nu)$, $p\geq1$ where $\nu$ is an invariant measure for the semigroup (see, for instance, \cite{DPZ96} and references therein).
Indeed, if $\mu$ is an invariant measure for the semigroup \eqref{e.L.38a}, then the semigroup  \eqref{e.L.38a} can be extended to a strongly continuous contraction semigroup in $L^p(H,\nu)$ whose generator is, say, $(K_p,D(K_p))$.
It worth to notice that as consequence of Theorem \ref{t.L.2}, the set $\mathcal E_A(H)$  is a  core (with respect to the norm of $L^p(H,\nu)$) for $(K_p,D(K_p))$.

Kolmogorov equations for measures in finite dimension have been the object of several papers.
We recall that in the papers \cite{BR00}  have been stated sufficient conditions in order to ensure existence of 
a weak solution for partial differential operators of the form 
$$
    H\varphi(t,x)= a^{ij}(t,x)\partial_{x_i}\partial_{x_j}\varphi(x)+b^i(t,x)\partial_{x_i}\varphi(x),\quad (t,x)\in (0,1)\times \Rset^d
$$
where $\varphi\in C_0^\infty( R^d) $ and $a^{ij},b^i\colon (0,1)\times R^d\to \Rset$, $1\leq i,j\leq d$ are suitable locally integrable functions.
With similar techniques, in  \cite{BDPR04} the problem is studied for parabolic differential operators of the form $Lu(t,x)=u_t(t,x)+Hu(t,x)$, $u\in  C_0^\infty((0,1)\times \Rset^d)$.
The infinite dimensional framework has been investigated  in \cite{BR01}, where it is considered an equation for measures of the form 
\[
  \int_X L_{A,B}\psi(x)\mu(dx)=0,\quad \forall \psi\in\mathcal K,
\]
where $X$ is a locally convex space, $\mathcal K$ is a suitable set of cylindrical functions and $ L_{A,B}$ is formally given by  
\[
 L_{A,B}\psi(x)=\sum_{i,j=1}^\infty A_{i,j}\partial_{e_i}\partial_{e_j}\psi(x)+\sum_{i =1}^\infty B_i\partial_{e_i}\psi(x),
\]
with $\mu$-measurable functions $A_{i,j}$, $B_i$ and vectors $e_i\in X$.
Under some integrability assumptions on   $A_{i,j}$, $B_i$, the authors prove existence of a measure $\mu$, possibly infinite, satisfying the above equation.

We stress that in our paper we prove {\em uniqueness} results.
In this direction, the results of   Theorems \ref{t.L.1.4}, \ref{t.RD.1.4} are, at our knowledge, new.

The paper is organized as follows:
in the next section we introduce the notions of $\pi$-convergence and we prove some results about the transition semigroup \eqref{e.L.38a} in the space $C_{b,1}(H)$.
Sections \ref{s.proof1}, \ref{s.proof2}, \ref{s.proof3} concern    proofs of Theorems \ref{t.L.1}, \ref{t.L.2}, \ref{t.L.1.4}, respectively.
Section \ref{s.application} is devoted to extend the results  to a reaction-diffusion operator.
\section{Notations and preliminary results}
If $E$ is a Banach space, we denote by $C_b(H;E)$ the Banach space of all uniformly continuous and bounded functions $f:H\to E$, endowed the usual supremum norm $\|\cdot\|_{C_b(H;E)}$. 
$C_b^1(H)$   denotes the space of all the functions $f\in C_b(H )$ which are Fr\'echet differentiable with uniformly continuous and bounded differential $DF\in C_b(H;\mathcal L(H;E))$.

We deal with semigroups of operators which  are not, in general, strongly continuous.
For this reason, we introduce the notion of $\pi$-convergence in the space $C_b(H)$ (see \cite{Manca07}, \cite{Priola}).

\begin{Definition}{\em
A sequence $(\varphi_n)_{n\in \Nset} \subset C_b(H)$ is said to be {\em $\pi$-convergent} to a function $\varphi$ $\in$ $C_b(H)$  if for any $x\in H$ we have 
$$
 \lim_{n\to \infty}\varphi_n(x)=\varphi(x)
$$ 
and 
$$
  \sup_{n\in\Nset}\|\varphi_n\|_0 <\infty. 
$$
Similarly,  the $m$-indexed sequence $(\varphi_{n_1,\ldots,n_m})_{n_1\in\Nset,\ldots,n_m\in\Nset}\subset C_b(H)$ is said to be $\pi$-convergent to  $\varphi$ $\in$ $C_b(H)$ if for any $i\in \{2,\ldots,m\}$  
 there exists an $i-1$-indexed sequence $(\varphi_{n_1,\ldots,n_{i-1}})_{n_1\in\Nset,\ldots,n_{i-1}\in\Nset}\subset C_b(H)$  such that
$$
 \lim_{n_i\to \infty}  \varphi_{n_1,\ldots,n_i} \stackrel{\pi}{=}\varphi_{n_1,\ldots,n_{i-1}} 
$$ 
and 
$$
 \lim_{n_1\to\infty}\varphi_{n_1 }\stackrel{\pi}{=}\varphi .
$$
We shall write
$$
 \lim_{n_1\to \infty}\cdots\lim_{n_m\to \infty} \varphi_{n_1,\ldots,n_m}\stackrel{\pi}{=}\varphi
$$ 
or $\varphi_n\stackrel{\pi}{\to}\varphi$ as $n \to \infty$, when the sequence has one index. 
}
\end{Definition}
Note that since the convergence is pointwise we can not take a diagonal sequence.
However, in order to avoid eavy notations, we shall often assume that the sequence has one index.

As easily seen the $\pi$-convergence implies the convergence in $L^p(H;\mu)$, for any $\mu\in{\cal M}(H)$, $p\in [1,\infty)$.

Let $k>0$. 
We shall often use the fact that if for a sequence  $(\varphi_n)_{n\in\Nset}\subset C_{b,k}(H)$ we have that
$(1+|x|^k)^{-1}\varphi_n\stackrel{\pi}{\to} \varphi\in    C_{b,k}(H)$ as $n\to \infty$, then the sequence converges to $\varphi$ in $L^p(H;\mu)$, for any $\mu\in{\cal M}_k(H)$, $p\in [1,\infty)$.
This argument may be viewed as an extension of the $\pi$-convergence to the spaces $C_{b,k}(H)$.

In Theorem \ref{t.L.2} we claim that   $\mathcal E_A(H) $    is a $\pi$-core for $( {K},D( {K}))$.
This means that if $\varphi\in\mathcal E_A(H)$ we have $\varphi\in D( {K}\sostituzioneA)$   and   $K\varphi=K_0\varphi$.
In addiction, if $\varphi \in D( {K}\sostituzioneA)$, there exist $m\in \Nset$ and an $m$-indexed sequence $(\varphi_{n_1,\ldots,n_m})_{n_1\in \Nset,\ldots,n_m\in \Nset}\subset \mathcal E_A(H)$ such that
\[
   \lim_{n_1\to \infty}\cdots\lim_{n_m\to \infty} \frac{\varphi_{n_1,\ldots,n_m}}{1+|\cdot|^{}}\stackrel{\pi}{=}\frac{\varphi }{1+|\cdot|^{} },
 \quad  \lim_{n_1\to \infty}\cdots\lim_{n_m\to \infty} \frac{K_0\varphi_{n_1,\ldots,n_m}}{1+|\cdot|^{} }\stackrel{\pi}{=}
      \frac{K\varphi}{1+|\cdot|^{}}.
\]
The construction of such a sequence is detailed   in section \ref{s.proof2}.

\subsection{The transition semigroup in $C_{b,1}(H)$}

This section is devoted to study the semigroup $  ({P}_t)_{t\geq0}$ in the space $C_{b,1}(H)$.
\begin{Proposition} \label{p.L.2.1}
 Formula \eqref{e.L.38a}  
defines a 
semigroup of operators $(  {P}_t)_{t\geq0} $ in $C_{b,1}(H)$,
and there exist  a family of probability measures $\{\pi_t(x,\cdot),\,t\geq0,\,x\in H\}\subset  \mathcal M_1(H)$ 
and two constants $c_0>0$, $\omega_0 \in \Rset$
such that
\begin{itemize}
\item[(i)] $ {P}_t\in \mathcal L( C_{b,1}(H))$ and $  \|  {P}_t\|_{\mathcal L(C_{b,1}(H))} \leq c_0 e^{ \omega_0t}$; 
\item[(ii)] $\DS  {P}_t\varphi(x)=\int_H\varphi(y)\pi_t(x,dy) $, for any $t\geq0$, $\varphi \in C_{b,1}(H)$, $x\in H$;
\item[(iii)] for any  $\varphi \in C_{b,1}(H)$, $x\in H$, the function $\Rset^+\to\Rset$, $t\mapsto  {P}_t\varphi(x)$ is continuous. 
\item[(iv)] $  {P}_t  {P}_s=  {P}_{t+s}$, for any $t,s\geq0$ and $  {P}_0=I$;
\item[(v)] for any $\varphi\in C_{b,1}(H)$ and any sequence $(\varphi_n)_{n\in \Nset}\subset C_{b,1} (H)$ such that 
\[
  \lim_{n\to\infty} \frac{\varphi_n}{1+|\cdot|^{}}  \stackrel{\pi}{=}\frac{\varphi }{1+|\cdot|^{}} 
\]
we have, for any $t\geq0$,  
\[
  \lim_{n\to\infty} \frac{{P}_t\varphi_n}{1+|\cdot|^{}}  \stackrel{\pi}{=}\frac{{P}_t\varphi }{1+|\cdot|^{}}. 
\]
\end{itemize}
\end{Proposition}
\begin{proof} 
 (i). Take   $\varphi\in   C_{b,1}(H)$, $t\geq0 $. 
We have to show that $ {P}_t\varphi\in C_{b,1}(H) $, that is the function $x\mapsto (1+|x|^{})^{-1}P_t\varphi(x) $ is uniformly continuous and bounded. 
Take $\varepsilon>0 $ and let $\theta_\varphi:\Rset^+\to \Rset $ be the modulus of continuity of $(1+|\cdot|^{})^{-1}\varphi$.
We have
\[
   \frac{ {P}_t\varphi(x)}{1+|x|^{}}- \frac{ {P}_t\varphi(y)}{1+|y|^{}}  
 = I_1(t,x,y)+I_2(t,x,y)+I_3(t,x,y), 
\]
where
\begin{eqnarray*}          
 I_1(t,x,y)&=& 
    \EE\left[\left( \frac{\varphi(X(t,x))}{1+|X(t,x)|^{}}- 
  \frac{\varphi(X(t,y))}{1+|X(t,y)|^{}}\right)
     \frac{1+ |X(t,x)|^{} }{1+|x|^{}} \right],\\
 I_2(t,x,y)&=& \EE\left[\frac{\varphi(X(t,y))}{1+|X(t,y)|^{}}
   \left(\frac{|X(t,x)|^{} -|X(t,y)|^{}}
                         {1+|x|^{}}  \right)\right],\\
  I_3(t,x,y)&=& \EE\left[\frac{\varphi(X(t,y))\left(1+|X(t,x)|^{}  \right)}{1+|X(t,y)|^{}}         
        \left( \frac{1}{1+|x|^{}} -\frac{1}{1+|y|^{}} \right)\right].
\end{eqnarray*} 
For $I_1(t,x,y)$ we have, by taking into account  \eqref{e.L.3}, \eqref{e.L.10a}, that there exists $c>0$ 
such that
\begin{eqnarray*}
  |I_1(t,x,y)|&\leq&    \EE\left[\theta_\varphi (|X(t,x)-X(t,y)| )   
  \frac{1+ |X(t,x)|^{} }{1+|x|^{}} \right] \\ 
&\leq&      \theta_\varphi \left(c_{}|x-y| \right) 
\frac{\EE\left[1+ |X(t,x)|^{}  \right]}{1+|x|^{}}\leq c_{}    \theta_\varphi (c_{}|x-y| ).
\end{eqnarray*}
Then, there exists $\delta_1>0$ such that $|I_1(t,x,y)|\leq \varepsilon/3$, for any $x,y\in H$ such that $|x-y|\leq \delta_1$.
For  $I_2(t,x,y)$ we have, by elementary inequalities,
\begin{eqnarray*}
  | I_2(t,x,y)|&\leq& \frac{\|\varphi\|_{0,1}}{1+|x|} \EE\left[ \left||X(t,x)|^{} -|X(t,y)|^{}  \right|\right] 
\\
   &\leq& \frac{ \|\varphi\|_{0,1}}{ 1+|x|^{}}   \EE\left[ |X(t,x)-X(t,y)|\right ] 
    \leq   \|\varphi\|_{0,1}  c_{} |x-y|.        
\end{eqnarray*} 
Then there exists $\delta_2>0$ such that $|I_2(t,x,y)|\leq \varepsilon/3$, for any $x,y\in H$ such that $|x-y|\leq \delta_2$.
Similarly, for $ I_3(t,x,y)$ we have
\begin{eqnarray*}
  | I_3(t,x,y)|&\leq& \|\varphi\|_{0,1}   \frac{ 1+\EE\left[ |X(t,x)|^{}  \right]}{1+|x|^{}}    \frac{\left||x|^{}-|y|^{}\right|}{   1+|y|^{} } 
\\
&\leq& c\|\varphi\|_{0,1}  (1+ c_{})   |x-y|.
\end{eqnarray*}
for some $c>0$.
Then, there exists $\delta_3>0$ such that $|I_3(t,x,y)|\leq \varepsilon/3$, for any $x,y\in H$ such that $|x-y|\leq \delta_3$.
Finally, for any $x,y\in H $ with $|x-y|\leq \min\{\delta_1,\delta_2,\delta_3   \}$ we find that 
\[
    \left|\frac{ {P}_t\varphi(x)}{1+|x|^{}}- \frac{ {P}_t\varphi(y)}{1+|y|^{}}  \right|<\varepsilon
\]
as claimed.
Now, by taking into account \eqref{e.L.10a}, there exists $c>0$ such that 
\[
     \left|\frac{ {P}_t\varphi(x)}{1+|x|^{}}\right|\leq \|\varphi\|_{0,1}\frac{1+\EE\left[ |X(t,x)|^{}\right]}{1+|x|^{}}
 \leq c\|\varphi\|_{0,1} 
\]
Then $ {P}_t\varphi\in C_{b,1}(H)$.
Note that by \eqref{e.L.10a} it follows that the operators $P_t$ are bounded in a neighborhood of $0$.
Hence,  the existence of the two constants $c_0>0$, $\omega_0\in\Rset$ follows by (iv) and by a standard argument. 
Notice that by the same argument follows\footnote{Of course, to prove (iv)-(v) we do not use this statement of (i)} (v).\\
(ii).
Take $\varphi\in C_{b,1}(H)$, and consider a sequence $(\varphi_n)_{n\in\Nset}\subset C_b(H)$ such that 
\begin{equation} \label{e.L.14}
    \lim_{n\to\infty}\frac{\varphi_n}{1+|\cdot|^{}}\stackrel{\pi}{=}\frac{\varphi}{1+|\cdot|^{}}.
\end{equation}
 Since $\pi_t(t,\cdot)$ is the image measure of $X(t,x)$ in $H$, the representation (ii) holds for any $\varphi_n$, that is
\[
   {P}_t\varphi_n(x)=\EE\bigl[\varphi_n(X(t,x))  \bigr]=\int_H\varphi_n(y)\pi_t(x,dy).
\]
Since \eqref{e.L.10a} holds we have $\pi(x,\cdot) \in \mathcal M_1(H)$,   and by \eqref{e.L.14} there exists $c>0$ such that $|\varphi_n(x)|\leq c(1+|x|^{}) $,   
for any $n\in \Nset$, $x\in H$.  
 Finally, the result follows by the dominated convergence theorem.  \\
(iii). For any $\varphi\in  C_{b,1}(H)$, $x\in H$, $t,s\geq0$ we have 
\begin{eqnarray*}
    {P}_t\varphi(x)- {P}_s\varphi(x)&=&
   \EE\left[\frac{\varphi(X(t,x))}{1+|X(t,x)|^{}}- \frac{\varphi(X(s,x))}{1+|X(s,x)|^{}}\left( 1+|X(t,x)|^{}  \right) \right] 
\\
  && +  \EE\left[    \frac{\varphi(X(s,x))}{1+|X(s,x)|^{}}\left(  |X(t,x)|^{}-|X(s,x)|^{}  \right) \right].
\end{eqnarray*}
Then
\begin{multline}  \label{e.L.15}
| {P}_t\varphi(x)- {P}_s\varphi(x)|\leq  
\EE\left[\theta_\varphi\left(  |X(t,x) - X(s,x) |   \right) \left( 1+|X(t,x)|^{}  \right) \right]\\
  +  \|\varphi\|_{0,1}\EE\left[  |X(t,x) - X(s,x) |\right],
\end{multline}
where $\theta_\varphi:\Rset^+\to \Rset$ is the modulus of continuity of $(1+|\cdot| )^{-1}\varphi $. 
Note also that  since for any $x\in H$ the process $(X(t,x))_{t\geq0} $ is continuous in mean square, we have  
\[
   \lim_{t\to s}  |X(t,x) - X(s,x) |    =0 \quad \PP\text{-a.s.}.
\]
Hence, by taking into account that $\theta_\varphi:\Rset^+\to \Rset$ is bounded and that  \eqref{e.L.10a} holds, we can apply the dominated convergence theorem to show that the first term in the right-hand side of \eqref{e.L.15} vanishes as $t\to s$.
Finally, the fact that the second term in the right-hand side of \eqref{e.L.15} vanishes as $t\to s$ may be proved by the same argument.
\\
(iv).
Take $\varphi\in C_{b,1}(H)$, and consider a sequence $(\varphi_n)_{n\in\Nset}\subset C_b(H)$ such that $(1+|\cdot|^{})^{-1} \varphi_n\stackrel{\pi}{\to}(1+|\cdot|^{})^{-1} \varphi$ as $n\to\infty$.
By the markovianity of the process $X(t,x)$ it follows that (iv) holds true for any $\varphi_n$.
Then, since by (iii) it follows that $(1+|\cdot|^{})^{-1} {P}_t\varphi_n\stackrel{\pi}{\to}(1+|\cdot|^{})^{-1} {P}_t\varphi$ as $n\to\infty$, still by (iii)  
we find
\[
 \frac{ {P}_{t+s}\varphi}{1+|\cdot|^{} }\stackrel{\pi}{=}  \lim_{n\to\infty}\frac{ {P}_{t+s}\varphi_n}{1+|\cdot|^{} }
  = \lim_{n\to\infty}\frac{ {P}_t {P}_s\varphi_n}{1+|\cdot|^{} }\stackrel{\pi}{=}\frac{ {P}_t {P}_s\varphi}{1+|\cdot|^{}}.
\]
This concludes the proof.
\end{proof}

\begin{Remark}{\em 
 We recall that for any $k>0 $, $T>0$ there exists $c_k>0$ such that 
\begin{equation*} 
  \sup_{t\in[0,T]}\EE[|X(t,x)|^k]<c_k(1+|x|^k), 
\end{equation*}
that implies $\{\pi_t(x,\cdot),\,t\geq0,\,x\in H\}\subset \bigcap_{k\geq0}\mathcal M_k(H)$.
Consequently,  all the results of this section are true with $C_{b,k}(H)$ replacing $C_{b,1}(H)$.
}
\end{Remark}

Here we collect some useful properties of the infinitesimal generator $(K,D(K\sostituzioneA)) $. 
\begin{Proposition} \label{p.L.2.9}
Let   $X(t,x)$  be the mild solution of problem \eqref{e.L.3.1} and let $(P_t)_{t\geq0}  $   
be the associated transition semigroups in the   space 
$C_{b,1}(H)$ defined by \eqref{e.L.38a}. 
Let also $(K,D(K\sostituzioneA)) $   
be the associated infinitesimal generators, defined by  \eqref{e.L.0}.  
Then
\begin{itemize}
\item[(i)] for any $\varphi \in D( {K})$, we have $  {P}_t\varphi \in D( {K}\sostituzioneA))$ and 
     $  {K} {P}_t\varphi =  {P}_t  {K}\varphi$, $t\geq0$;
\item[(ii)] for any $\varphi \in D( {K}\sostituzioneA)$, $x\in H$, the map $[0,\infty)\to \Rset$, $t\mapsto  {P}_t\varphi(x)$ is continuously differentiable and $(d/dt) {P}_t\varphi(x) =  {P}_t  {K}\varphi(x)$;
\item[(iii)]   given $c_0>0$ and $\omega_0$ as in Proposition \ref{p.L.2.1},  for any $\lambda>\omega_0$ the linear operator $R(\lambda, {K})$ on $C_{b,1}(H)$ done by
$$
   R(\lambda, {K})f(x)=\int_0^\infty e^{-\lambda t} {P}_tf(x)dt, \quad f\in C_{b,1}(H),\,x\in H
$$
satisfies, for any  $f\in C_{b,1}(H)$
$$
   R(\lambda, {K})\in \mathcal L(C_{b,1}(H)),\quad \quad \|R(\lambda, {K})\|_{\mathcal L(C_{b,1}(H))}\leq 
\frac{c_0 }{\lambda-\omega_0} 
$$
$$
  R(\lambda, {K})f\in D( {K}\sostituzioneA),\quad (\lambda I- {K})R(\lambda, {K})f=f.
$$
We call $R(\lambda, {K})$ the {\em resolvent} of $  {K}$ at $\lambda$.
\end{itemize}
\end{Proposition}
\begin{proof}
(i). 
It is proved by taking into account \eqref{e.L.0} and (iii) of Proposition \ref{p.L.2.1}.\\
(ii). This follows easily by (i)  and by (iii) of Proposition \ref{p.L.2.1}. \\
(iii). By (i) of Proposition \ref{p.L.2.1} we have
\[
   \left\| \int_0^\infty e^{-\lambda t} {P}_tfdt \right\|_{0,1}\leq c_0\int_0^\infty e^{-(\lambda-\omega_0 )t}dt\| f \|_{0,1}=
 \frac{c_0\| f \|_{0,1}}{\lambda-\omega_0}.
\]
Finally, the fact that  $R(\lambda, {K})f\in D( {K})$ and $ (\lambda I- {K})R(\lambda, {K})f=f$ hold   can be proved in a standard way (see, for instance, \cite{Cerrai}, \cite{Priola}).
\end{proof}

\section{Proof of Theorem \ref{t.L.1}} \label{s.proof1}
In order to prove this theorem, we need some results about the transition semigroup $(P_t)_{t\geq0} $ in the space $C_b(H)$.
Denote by $\pi_t(x,\cdot)$ the image measure of the mild solution $X(t,x)$ of problem \eqref{e.L.3.1}.
Since for any $\varphi\in C_b(H)$ the representation 
\[
 P_t\varphi(x)=\int_H\varphi(y)\pi_t(x,dy),\quad x\in H,\, t\geq0
\]
holds (cf (ii) of Proposition \ref{p.L.2.1}) and $X(t,x)$ is  continuous in mean square, it follows easily that  $(P_t)_{t\geq0} $ is a semigroup of contraction operators in the space $C_b(H) $.  
Moreover, we have that for any $x\in H$, $\varphi\in C_b(H)$ the function $\Rset^+\to \Rset$, $t\mapsto P_t\varphi(x)$ is continuous (cf (iii) of Proposition \ref{p.L.2.1}).
This means that $(P_t)_{t\geq0} $ is {\em stochastically continous Markov semigroup}, in the sense introduced in \cite{Manca07}.

We denote by $(K,D(K,C_b(H))$ the infinitesimal generator of $P_t$ is the space $C_b(H)$, defined by
\begin{equation}  \label{e.L.0a}
\begin{cases}
\DS D(K,C_b(H))=\bigg\{ \varphi \in C_b(H): \exists g\in C_b(H), \lim_{t\to 0^+} \frac{P_t\varphi(x)-\varphi(x)}{t}=g(x),  \\ 
  \DS  \qquad\qquad\quad x\in H,\;\sup_{t\in(0,1)}\left\|\frac{P_t\varphi-\varphi}{t}\right\|_0<\infty \bigg\}\\   
   {}   \\
  \DS K\varphi(x)=\lim_{t\to 0^+} \frac{P_t\varphi(x)-\varphi(x)}{t},\quad \varphi\in D(K,C_b(H)),\,x\in H.
\end{cases}
\end{equation}
It is clear that $D(K,C_b(H))\subset D(K\sostituzioneA)$.
The key result we use to prove the Theorem is the following, proved in \cite{Manca07}
\begin{Theorem} \label{t.L.2.2}
  For any $\mu\in\mathcal M (H)$ there exists a unique family of measures $\{\mu_t\}_{t\geq0}\subset \mathcal M (H)$ 
such that
\begin{equation} \label{e.L.14a}
  \int_0^T|\mu_t|(H)dt <\infty,\quad \forall T>0
\end{equation}
and  
\begin{equation} \label{e.L.11a}
   \int_H\varphi(x)\mu_t(dx)- \int_H\varphi(x)\mu(dx)=\int_0^t\left( \int_HK\varphi(x)\mu_s(dx)   \right)ds
\end{equation}
 holds for any $t\geq0,\,\varphi\in D(K,C_b(H))$.
\end{Theorem}

We split the proof into two lemmata.

\begin{Lemma} \label{l.L.3}
 The formula
\begin{equation} \label{e.L.11}
   \langle \varphi, {P}_t^*F\rangle_{\mathcal L(C_{b,1}(H),(C_{b,1}(H))^*)}=
    \langle  {P}_t\varphi,F\rangle_{\mathcal L(C_{b,1}(H),(C_{b,1}(H))^*)}
\end{equation}
defines a semigroup of linear operators in  $(C_{b,1}(H))^*$. 
Finally, ${P}_t^*: \mathcal M_1(H)\to\mathcal M_1(H)$ and it maps positive measures into positive measures.
\end{Lemma}
\begin{proof}
Fix $t\geq0$. 
By \eqref{e.L.10a} it follows that there exists $c>0$ such that $| {P}_t\varphi(x)|\leq c\|\varphi\|_{0,1}(1+|x|^{}) $, for any $\varphi \in C_{b,1}(H)$.
Then, if $F \in (C_{b,1}(H))^*$, we have
\[
     \left|\langle \varphi, {P}_t^*F\rangle_{\mathcal L(C_{b,1}(H),(C_{b,1}(H))^*)} \right|\leq c\|F\|_{(C_{b,1}(H))^*}\|\varphi\|_{0,1},
\]
for any $\varphi \in C_{b,1}(H)$. 
Since $ {P}_t^*$ is linear, it follows that $ {P}_t^*\in \mathcal L((C_{b,1}(H))^* )$.
Note that by (ii) of Proposition \ref{p.L.2.1} it follows ${P}_t\varphi\geq0$ for any $\varphi\geq0$. 
This  implies that if $\langle \varphi,F\rangle\geq0$ for any $\varphi\geq0$, then $\langle \varphi, {P}_t^*F\rangle\geq0$ for any $\varphi\geq0 $.
Hence, in order to check that     ${P}_t^*: \mathcal M_1(H)\to\mathcal M_1(H)$, it is sufficient to take $\mu$ positive. 
So, let $\mu \in \mathcal M_1(H)$ be   positive and consider the map 
\[
 \Lambda:\mathcal B(H)\to \Rset,\quad  \Gamma\mapsto \Lambda(\Gamma)=\int_H\pi_t(x,\Gamma)\mu(dx). 
\]
We recall that since $X(t,x)$ is continuous with respect to  $x$, for any $\Gamma\in \mathcal B(H)$ the map $H\to [0,1]$, $x\to \pi_t(x,\Gamma)$ is Borel, and consequently the formula above in meaningful.
It is straightforward to see that $\Lambda$ is a positive and finite Borel measure on $H$, namely $\Lambda\in \mathcal M(H)$. 
We now show $\Lambda=P_t^*\mu$.

Let us fix $\varphi\in C_b(H)$, and consider a sequence of simple Borel functions $(\varphi_n)_{n\in \Nset}$ which converges uniformly to $\varphi$ and such that $|\varphi_n(x)|\leq |\varphi(x)|$, $x\in H$.
For any $x\in H$ we have
\[
  \lim_{n\to\infty}\int_H \varphi_n(y)\pi_t(x,dy)=\int_H \varphi (y)\pi_t(x,dy)=P_t\varphi(x)
\]
and 
\[
    \sup_{x\in H}\left|\int_H \varphi_n(y)\pi_t(x,dy)   \right| \leq \|\varphi\|_0.
\]
Hence, by the dominated convergence theorem and by taking into account  that $\varphi_n$ is simple, we have
\begin{eqnarray*}
   \int_H\varphi(x)\Lambda(dx)&=&  \lim_{n\to\infty}\int_H\varphi_n(x)\Lambda(dx)=\lim_{n\to\infty}\int_H\left(\int_H \varphi_n(y)\pi_t(x,dy) \right)\mu(dx)    \\
&=& \int_H\left(\int_H \varphi (y)\pi_t(x,dy) \right)\mu(dx)= \int_H P_t\varphi(x)\mu(dx).
\end{eqnarray*}
This implies that $P_t^*\mu=\Lambda$ and consequently $P_t^*\mu\in \mathcal M(H)$. 

In order to show that  $P_t^*\mu\in \mathcal M_1(H)$, consider a sequence of functions $(\psi_n)_{n\in \Nset}\subset C_b(H)$ such that $\psi_n(x)\to |x|$ as $n\to \infty$ and $ \psi(x)\leq |x|$, for any $x\in H$.
By Proposition \ref{p.L.2.1} we have $\int_H\psi_n(y)\pi_t(x,dy)\to  \int_H|y|\pi_t(x,dy)$ as $n\to \infty$ and $\int_H\psi_n(y)\pi_t(x,dy)  \leq c(1+|x|)$, for any $x\in H$ and for some $c>0$.
Hence, since $\mu\in \mathcal M_1(H)$ we have
\begin{eqnarray*}
  &&\int_H |x|P_t^*\mu(dx)=\lim_{n\to\infty} \int_H \psi_n(x) P_t^*\mu(dx)\\
 &&\qquad=\lim_{n\to\infty} \int_H \left(\int_H \psi_n (y)\pi_t(x,dy) \right)\mu(dx) \leq \int_H c(1+|x|)\mu(dx)<\infty
\end{eqnarray*}
This concludes the proof.
\end{proof}

\begin{Lemma} \label{l.L.5.3}
For any  $\mu\in\mathcal M_1(H)$ there exists a unique family of finite measures $\{\mu_t\}_{t\geq0} \subset \mathcal M_1(H)$
fulfilling  \eqref{e.L.5b}, \eqref{e.L.8}, and this family is given by $\{ {P}_t^*\mu\}_{t\geq0} $.
\end{Lemma}
\begin{proof}
We first check that $\{ {P}_t^*\mu\}_{t\geq0} $ satisfies \eqref{e.L.5b}, \eqref{e.L.8}.
 By Proposition \ref{l.L.3}, for any $\mu\in \mathcal M_1(H)$,  formula \eqref{e.L.11} defines a family $\{ {P}_t^*\mu\}_{t\geq0}$ of measures  on $H$.
Moreover, by (i) of Proposition \ref{p.L.2.1} it follows that for any $T>0$ it holds
\[
 \sup_{t\in[0,T]}\| {P}_t^*\mu\|_{(C_{b,1}(H))^*}  \sup_{t\in[0,T]}\int_H(1+|x|)|P_t^*\mu|(dx) = <\infty.
\]
Hence, \eqref{e.L.5b} holds.
We now show \eqref{e.L.8}.
By (i), (ii), (iv) of Proposition \ref{p.L.2.1} and by the dominated convergence theorem it follows easily 
that for any  $\varphi \in C_{b,1}(H)$ the function
\begin{equation}  \label{e.L.7}
 \Rset^+ \to \Rset,\quad t\mapsto \int_H\varphi(x) {P}_t^*\mu(dx)
\end{equation}
is continuous.
Clearly, $ {P}_0^*\mu=\mu$.
Now we show that if $\varphi \in D( {K}\sostituzioneA)$ then the function \eqref{e.L.7} is differentiable.
Indeed, by taking into account \eqref{e.L.0} and (i) of Proposition \ref{p.L.2.9}, 
for any  $\varphi \in D( {K}\sostituzioneA)$  we can apply the dominated convergence theorem to obtain 
\begin{eqnarray*}
  &&\frac{d}{dt}\int_H \varphi(x) {P}_t^*\mu(dx)=\\
  &&\qquad = \lim_{h\to 0}\frac1h \left(\int_H  {P}_{t+h}\varphi(x)\mu(dx)- \int_H  {P}_t\varphi(x)\mu_t(dx)\right)  \\
&&\qquad=\lim_{h\to 0}\int_H \left(\frac{ {P}_{t+h}\varphi(x) -  {P}_t\varphi(x)}{h}\right)\mu(dx)\\
 &&\qquad =\lim_{h\to 0}\int_H  {P}_t\left(\frac{ {P}_h\varphi - \varphi}{h}\right)(x)\mu(dx)\\
 &&\qquad =\int_H \lim_{h\to 0}\left(\frac{ {P}_h\varphi - \varphi}{h}\right)(x) {P}_t^*\mu(dx)= 
  \int_H  {K}\varphi(x) {P}_t^*\mu(dx).
\end{eqnarray*}
Then, by arguing as above, it follows that the differential of the mapping defined by \eqref{e.L.7} is continuous. 
This clearly implies that   $\{ {P}_t^*\mu\}_{t\geq0}$ satisfies \eqref{e.L.8}.
In order to show uniqueness of such a solution, by the linearity of the problem it is sufficient to show that if $\mu=0$ and $\{\mu_t\}_{t\geq0}\subset \mathcal M_1(H) $ is a solution of equation \eqref{e.L.8}, then $\mu_t=0$, for any $t\geq0$.
Note that equation \eqref{e.L.8} holds in particular for $\varphi\in D(K,D(K))$ (cf \eqref{e.L.0a}) and consequently \eqref{e.L.11a} holds, for any $\varphi\in D(K,D(K))$.
Note also that by \eqref{e.L.5b} follows that  $\{  *\mu_t\}_{t\geq0}$ fulfils \eqref{e.L.14a}.
Hence, by Theorem \ref{t.L.2.2}, it follows that $\mu_t=0$, $\forall t\geq0$.
This concludes the proof. 
\end{proof}
\section{Proof of Theorem \ref{t.L.2} } \label{s.proof2}
We split the proof in several steps.
We start by studying the  Ornstein-Uhlenbeck operator in $C_{b,1}(H)$ that is, roughly speaking, the case $F=0$ in \eqref{e.4}.
In Proposition \ref{p.L.4.3} we shall prove Theorem \ref{t.L.1.4} in the case $F=0$.
Then, Corollary \ref{c.L.4.6} will show that $(K,D(K_0))$ is an extension of $K_0$ and $K\varphi=K_0\varphi$ for any $\varphi\in \mathcal E_A(H)$.
In order to complete the proof of the theorem, we  shall present several approximation results.
Finally, Lemma \ref{l.L.4.6}  will complete the proof. 

\subsection{The Ornstein-Uhlenbeck semigroup in $C_{b,1}(H)$} \label{s.L.OU}
An important role in what follows it is played by the {\em Ornstein-Uhlenbeck semigroup} $(R_t)_{t\geq0}$ in the space $C_{b,1}(H)$, defined by the formula
\[
   R_t\varphi(x)= \begin{cases}
                  \DS \varphi(x), &t=0,\\
                  \DS        \int_H \varphi(e^{tA}x+y)N_{Q_t}(dy),&   t>0
                         \end{cases}
\]
where $\varphi\in C_{b,1}(H)$, $x\in H$ and $N_{Q_t}$ is the gaussian measure of zero mean and covariance operator $Q_t$ (cf Hypothesis \ref{h.L.3.1}).
It is well known that the representation 
\begin{equation} \label{e.L.43b}
    R_t\varphi(x)=\EE\left[\varphi\left( e^{tA}x+\int_0^te^{(t-s)A}BdW(s) \right)\right] 
\end{equation}
holds, for any $t\geq0$, $\varphi\in C_{b,1}(H)$, $x\in H$.
Hence, the  Ornstein-Uhlenbeck semigroup $(R_t)_{t\geq0}$ is the transition semigroup \eqref{e.L.38a}  in the case $F=0$ in \eqref{e.L.3.1}. 
Consequently,    $(R_t)_{t\geq0}$ satisfies stamentes (i)--(v) of  Proposition \ref{p.L.2.1}.
It is well known the following identity
 \begin{equation} \label{e.L.43}
  R_t(e^{i\langle \cdot,h\rangle})(x)=e^{i\langle
e^{tA}x,h \rangle -\frac12 \langle Q_th,h\rangle},
\end{equation}
which   implies  $R_t: \mathcal E_A(H)\to \mathcal E_A(H)$, for any $t\geq0$.
We  define the infinitesimal generator $L:D(L\sostituzioneA)\to C_{b,1}(H)$ of $(R_t)_{t\geq0}$ in $C_{b,1}(H)$ as in 
\eqref{e.L.0}, with $L$ replacing ${K} $ and $R_t$ replacing ${P}_t$.

\begin{Theorem} \label{t.L.4.1}
 Let $(P_t)_{t\geq0}$ be the semigroup \eqref{e.L.38a} and let $(R_t)_{t\geq0}$ be the Ornstein-Uhlenbeck semigroup \eqref{e.L.43b}.
We denote by $(K,D(K\sostituzioneA)) $, $ (L,D(L\sostituzioneA)) $ the corresponding infinitesimal generators in $C_{b,1}(H)$.
Then we have  $D(L\sostituzioneA)\cap C_b^1(H)=D(K\sostituzioneA)\cap C_b^1(H) $ and $K\varphi=L\varphi+\langle D\varphi,F\rangle$, for any
$\varphi \in  D(L\sostituzioneA)\cap C_b^1(H)$.
\end{Theorem}
\begin{proof}
Let $X(t,x)$ be the mild solution of  equation \eqref{e.L.3.1} and let us set 
$$
    Z_A(t,x)=e^{tA}+\int_0^te^{(t-s)A}BdW(s).
$$ 
Take $\varphi \in D(L\sostituzioneA)\cap C_b^1(H)$.
By taking into account that  
$$
   X(t,x)=Z_A(t,x)+ \int_0^te^{(t-s)A}F(X(s,x))ds,
$$
by the Taylor formula we have that $\PP$-a.s. it holds
$$
  \varphi(Z_A(t,x))
   =\varphi(Z_A(t,x))-\varphi(X(t,x))+\varphi(X(t,x))
$$
$$ 
 =\varphi(X(t,x)) -\int_0^1\left\langle D\varphi(\xi Z_A(t,x)+(1-\xi)X(t,x)),\int_0^te^{(t-s)A}F(X(s,x))ds\right\rangle d\xi .
$$
Then we have
$$
   R_t\varphi(x)-\varphi(x) = \EE\big[\varphi(Z_A(t,x))\big]-\varphi(x)= P_t\varphi(x)-\varphi(x)
$$
$$
   -\EE\left[\int_0^1\left\langle D\varphi(\xi Z_A(t,x)+(1-\xi)X(t,x)),\int_0^te^{(t-s)A}F(X(s,x))ds\right\rangle d\xi \right].
$$
Since $\varphi\subset D(L\sostituzioneA)\cap C_b^1(H)$,
it follows easily that for any $x\in H$
$$
  \lim_{t\to 0^+} \frac{P_t\varphi(x)-\varphi(x)}{t} = L\varphi(x)+\langle D\varphi(x),F(x)\rangle
$$
and 
\[
   \sup_{t\in(0,1]}\bigg\|  \frac{P_t\varphi-\varphi}{t}\bigg\|_{0,1} 
\]
\[
 \leq \sup_{t\in(0,1)}  \bigg\|  \frac{R_t\varphi-\varphi}{t}\bigg\|_{0,1} +\\
\sup_{x\in H} \|D\varphi(x)\|_{\mathcal L(H)}  \sup_{x\in H}\frac{|F(x)|}{1+|x|}   <\infty,
\]
that implies $\varphi \in D(K\sostituzioneA)$ and $K\varphi=L\varphi+\langle D\varphi ,F \rangle$.
The opposite inclusion follows by interchanging the role of $R_t$ and $P_t$ in the Taylor formula.
\end{proof}
We need the following approximation result, proved in  \cite{Manca07}. 
\begin{Proposition} \label{p.L.6.2}
For any $\varphi\in C_b(H)$, there exists $m\in\Nset$ and an $m$-indexed sequence $(\varphi_{n_1,\ldots,n_m})_{n_1,\ldots,n_m\in\Nset}\subset \mathcal E_A(H)$ such that 
\begin{equation} \label{e.L.20}
  \lim_{n_1\to \infty}\cdots  \lim_{n_m\to \infty}\varphi_{n_1,\ldots,n_m}\stackrel{\pi}{=}\varphi.
\end{equation}
Moreover, if $\varphi\in C_b^1(H)$, we can choose the  sequence  in such a way that \eqref{e.L.20} holds and 
\[
   \lim_{n_1\to \infty}\cdots  \lim_{n_m\to \infty} \langle D \varphi_{n_1,\ldots,n_m},h\rangle \stackrel{\pi}{=}\langle D \varphi ,h\rangle,
\]
for any $h\in H$.
\end{Proposition}
Now we are able to prove the following
\begin{Proposition} \label{p.L.4.3}
 For any $\varphi \in \mathcal E_A(H)$ we have   
$\varphi\in D(L\sostituzioneA)$ and 
\begin{equation} \label{e.L.44}
     L\varphi(x)= \frac12\textrm{Tr}\big[BB^*D^2\varphi(x)\big]+\langle x,A^*D\varphi(x)\rangle,\quad x\in H.  
\end{equation}
The set $\mathcal E_A(H)$ is a $\pi$-core for 
$ (L,D(L\sostituzioneA))$, and for any $\varphi\in D(L\sostituzioneA)\cap C_b^1(H) $  there exists $m\in\Nset$ and an $m$-indexed sequence $(\varphi_{n_1,\ldots,n_m})_{n_1,\ldots,n_m\in\Nset}\subset \mathcal E_A(H)$ such that 
\begin{equation} \label{e.L.45} 
  \lim_{n_1\to \infty}\cdots  \lim_{n_m\to \infty}
     \frac{\varphi_{n_1,\ldots,n_m}}{1+|\cdot|}\stackrel{\pi}{=}\frac{\varphi}{1+|\cdot|},
\end{equation}
\begin{equation} \label{e.L.46} 
   \lim_{n_1\to \infty}\cdots  \lim_{n_m\to \infty} 
\frac{\frac12\textrm{Tr}\big[BB^*D^2\varphi_{n_1,\ldots,n_m}\big]+\langle \cdot ,A^*D\varphi_{n_1,\ldots,n_m}\rangle
}{1+|\cdot|} \stackrel{\pi}{=}\frac{L\varphi}{1+|\cdot|}.
\end{equation}
Finally, if $\varphi\in D(L\sostituzioneA)\cap C_b^1(H) $ we can choose the sequence in such a way that \eqref{e.L.45}, \eqref{e.L.46} hold and
\begin{equation} \label{e.L.47} 
\lim_{n_1\to \infty}\cdots  \lim_{n_m\to \infty} 
\langle D \varphi_{n_1,\ldots,n_m},h\rangle \stackrel{\pi}{=}\langle D \varphi ,h\rangle,
\end{equation}
for any $h\in H$.
\end{Proposition}
\begin{proof}
We recall that the proof of \eqref{e.L.44} may be found in \cite{DP04}, Remark 2.66 (in \cite{DP04} the result is proved for the semigroup $(R_t)_{t\geq0}$ in the space $C_{b,2}(H)$, but it is clear that the result holds also in $C_{b,1}(H)$).

 Here we give only a sketch of the proof, which is very similar to the proof given in \cite{Manca07}.
 Take $\varphi\in D(L\sostituzioneA)$. 
For any $n_2\in \Nset$, set $\varphi_{n_2}(x)= n_2\varphi(x)/(n_2+|x|^2)$.
Clearly, $\varphi_{n_2}\in C_b(H)$ and $(1+|\cdot|)^{-1}\varphi_{n_2}\stackrel{\pi}{\to}(1+|\cdot|)^{-1}\varphi$ as $n_2\to \infty$.
By Proposition \ref{p.L.6.2}, for any $n_2\in \Nset$ we fix a sequence\footnote{we assume that the sequence has only one index} $(\varphi_{n_2,n_3})_{n_3\in \Nset}\subset \mathcal E_A(H)$ such that $\varphi_{n_2,n_3}\stackrel{\pi}{\to}\varphi_{n_2}$ as $n_3\to \infty$.
Set now, for any $n_1,n_2,n_3,n_4\in \Nset$ 
\begin{equation} \label{e.L.48}
   \varphi_{n_1,n_2,n_3,n_4}(x)=\frac{1}{n_4}\sum_{k=1}^{n_4}R_{\frac{k}{n_1n_4}}\varphi_{n_2n_3}(x).
\end{equation}
Since for any $\varphi\in C_{b,1}(H)$, $x\in H$ the function $\Rset^+\to \Rset$, $t\mapsto R_t\varphi(x)$ is continuous, a straightforward computation shows that the sequence $(\varphi_{n_1,\ldots,n_4})$ fulfils \eqref{e.L.45}.
Similarly, we find that for any $x\in H$ it holds
\begin{eqnarray*}
  \lim_{n_1\to\infty}\lim_{n_2\to\infty}\lim_{n_3\to\infty}\lim_{n_4\to\infty}
   &&\frac12\textrm{Tr}\big[BB^*D^2\varphi_{n_1,n_2,n_3,n_4}(x)\big]+\langle x ,A^*D\varphi_{n_1,n_2,n_3,n_4}(x)\rangle 
\\
  &=& \lim_{n_1\to\infty}\lim_{n_2\to\infty}\lim_{n_3\to\infty}\lim_{n_4\to\infty}L\varphi_{n_1,n_2,n_3,n_4}(x)
\\
  &=&  \lim_{n_1\to\infty}\lim_{n_2\to\infty}\lim_{n_3\to\infty}n_1\int_0^{\frac{1}{n_1}}LR_t\varphi_{n_2,n_3}(x)dt
\\
 &=& \lim_{n_1\to\infty}\lim_{n_2\to\infty}\lim_{n_3\to\infty}n_1\left(R_{\frac{1}{n_1}}\varphi_{n_2,n_3}(x)-\varphi_{n_2,n_3}(x)  \right)
\\
  &=& \lim_{n_1\to\infty} \left(R_{\frac{1}{n_1}}\varphi (x)-\varphi (x)  \right)=L\varphi(x).
\end{eqnarray*}
Here we have used the continuity of $t\mapsto LR_t\varphi_{n_2,n_3}(x)$ and the fact that $LR_t\varphi_{n_2,n_3}(x)=(d/dt)R_t\varphi_{n_2,n_3}(x)$ (cf Proposition \ref{p.L.2.1} and Proposition \ref{p.L.2.9}).
The fact that any limit above is equibounded in $C_{b,1}(H)$ with respect to the corresponding index follows by the construction of $\varphi_{n_1,n_2,n_3,n_4}(x)$.
Hence, \eqref{e.L.46} follows.

If $\varphi\in D(L\sostituzioneA)\cap C_b^1(H) $, by Proposition \ref{p.L.6.2},  there exists a sequence\footnote{we assume that the sequence has only one index} $(\varphi_{n})_{n\in \Nset}\subset \mathcal E_A(H)$ such that $\varphi_{n}\stackrel{\pi}{\to}\varphi$ as $n\to \infty$ and  $\langle D\varphi_{n },h\rangle\stackrel{\pi}{\to}\langle D\varphi ,h\rangle$ as $n\to\infty$, for any $h\in H$.
Since  for any $t>0$, $n\in\Nset$ we have
\[
   \langle DR_t\varphi_n(x),h\rangle= \int_H \langle D\varphi_n(e^{tA}x+y),h\rangle N_{Q_t}(dy),\quad x\in H
\]
it follows  $\langle DR_t\varphi_{n },h\rangle\stackrel{\pi}{\to}\langle DR_t\varphi ,h\rangle$ as $n\to\infty$, for any $h\in H$.
Then, the claim follows by arguing as above.
\end{proof}
By Theorem \ref{t.L.4.1} and Proposition \ref{p.L.4.3} we have
\begin{Corollary} \label{c.L.4.6}
 $(K,D(K\sostituzioneA))$ is an extension of $K_0$ and for any 
    $\varphi \in \mathcal E_A(H)$ we have $\varphi\in D( {K}\sostituzioneA)$ and $  {K}\varphi=K_0\varphi$. 
\end{Corollary}

\subsection{Approximation of $F$ with smooth functions}
It is convenient to introduce an auxiliary Ornstein--Uhlenbeck
semigroup
$$
   U_t\varphi(x)=\int_H\varphi(e^{tS}x+y)N_{\frac{1}{2}\;S^{-1}(e^{2t S}-1)}(dy),\quad \varphi\in C_b(H)
$$
where $S:D(S)\subset H\to H$ is a self--adjoint negative definite operator such 
that $S^{-1}$ is of trace class.
We notice  that $U_t$ is strong Feller, and  for any $t>0$, $\varphi\in C_b(H)$,  $U_t\varphi$ is infinite times differentiable with bounded differentials (see \cite{DP04}).
We introduce a   regularization 
of $F$  by setting
\begin{equation*}    
\langle F_n(x),h\rangle= 
\int_H\left\langle F\left(e^{\frac{1}{n} 
S}x+y\right),e^{\frac{1}{n} S}h\right\rangle N_{\frac{1}{2}\;S^{-1}(e^{\frac{2}{n} S}-1)}(dy),\quad n\in \Nset.
\end{equation*}
It is easy to check that $F_n$ is 
infinite times differentiable, with first differential bounded by $\kappa$, for any $n\in \Nset$. 
Moreover, $F_n(x)\to F(x)$ as $n\to\infty$ for all $x\in
H$ and $|F_n(x)|\leq |F(x)|$, for all $n\in\Nset$, $x\in
H$.

Let $P_t^n$ be the transition semigroup
\begin{equation} \label{e.L.17}
  P_t^n\varphi(x)=\EE[\varphi(X^n(t,x))],\quad \varphi\in C_{b,1}(H)
\end{equation}
where $X^n(t,x)$   is the solution of \eqref{e.L.3.1} with $F_n$ replacing $F$.
It is easy to check   
\[
    \lim_{n\to \infty }  \EE\bigl[ |  X^n(t,x)-X(t,x)|^2 \bigr]=0,\quad t\geq0,\,x\in H
\]
and 
\begin{equation*} 
   \EE\bigl[|X^n(t,x)|   \bigr]\leq   \EE\bigl[|X (t,x)|   \bigr], 
     \quad t\geq0,\, x\in H,
\end{equation*}
where $c_0>0$, $\omega_0\in \Rset$ are as in Proposition \ref{p.L.2.1}. 
This implies
\begin{equation} \label{e.L.28}
  \lim_{n\to\infty}  \frac{P_t^n\varphi}{1+|\cdot|} \stackrel{\pi}{=} \frac{P_t \varphi}{1+|\cdot|},
\end{equation}
for any $t\geq0$, $\varphi\in C_{b,1}(H)$.
We denote by $(K_n,D(K_n\sostituzioneA))$ the infinitesimal generator of the semigroup $P_t^n$ in $C_{b,1}(H)$, defined as in \eqref{e.L.0} with $K_n$ replacing $K$ and $P^n_t$ replacing $P_t$.
We recall that all the statements of Proposition \ref{p.L.2.1}, Theorem \ref{t.L.2.2}  hold also for $P_t^n$ and $(K_n,D(K_n\sostituzioneA))$. 
We also recall that the resolvent of $(K,D(K\sostituzioneA))$ in $C_{b,1}(H)$ is defined for any $\lambda>\omega_0$ by the formula $R(\lambda,K)f(x)=\int_0^\infty e^{-\lambda t}P_tf(x)dt$, $f\in C_{b,1}(H)$, $x\in H$ (cf Theorem \ref{t.L.2.2}). 
Similarly, for a fixed $n\in \Nset$ the resolvent of $(K_n,D(K_n\sostituzioneA))$ in $C_{b,1}(H)$ at $\lambda>0$  is defined by the same formula with  $P_t^n$ replacing $P_t$.   
Since  \eqref{e.L.28} holds, it is straightforward to see that
\begin{equation} \label{e.L.19}
   \lim_{n\to\infty} \frac{R(\lambda,K_n)\varphi}{1+|\cdot|} \stackrel{\pi}{=}\frac{R(\lambda,K )\varphi}{1+|\cdot|},
\end{equation}
for any $\varphi\in C_{b,1}(H)$, $\lambda>\omega_0$. %
%
%
%
%

%
%


The following proposition follows by Corollary 4.9 of \cite{Manca07} and by the fact that $\|DF_n\|\leq \kappa$, for any $n\in \Nset$.

\begin{Proposition} \label{P.L.7.2}
For any $n\in \Nset$, let $(K_n,D(K_n) )$ be the infinitesimal generator of the semigroup \eqref{e.L.17}.
Then for any $\lambda> \max\{0, \omega+M\kappa\} $, the resolvent $R(\lambda,K_n) $ of $K_n$ at $\lambda$ maps $C_b^1$ into $C_b^1(H)$ and it holds
\begin{equation} \label{e.L.43a}
  \|DR(\lambda,K_n)f\|_{C_b(H;H)}\leq \frac{M\|Df\|_{C_b(H;H)}}{\lambda-(\omega+M\kappa)},\quad f\in C_b^1(H).
\end{equation}
\end{Proposition}
Corollary \ref{c.L.4.6} shows that  ${K}$ is an extension of $K_0$ and that ${K}\varphi=K_0\varphi$, $\forall \varphi \in \mathcal E_A(H)$.
So, in view of the fact that ${K} {P}_t\varphi= {P}_tK_0\varphi$ for any $\varphi \in  \mathcal E_A(H)$ (cf (i) of Proposition \ref{p.L.2.9}), it is not difficult to check that $\{ {P}^*_t\mu\}_{t\geq0}$ fulfils \eqref{e.L.5b}, \eqref{e.L.2.5a}.
Now, let $\mu\in \mathcal M_1(H)$ and assume that $\{\mu_t\}_{t\geq0} \subset \mathcal M_1(H)$ fulfils 
\eqref{e.L.5b}, \eqref{e.L.2.5a}. 
In view of Theorem \ref{l.L.5.3}, to prove that $\mu_t= {P}^*_t\mu$, for any $t\geq0$, it is sufficient to show that $\{\mu_t\}_{t\geq0} $ is also a solution of \eqref{e.L.8}.
In order to do this,  we need an approximation result.
\begin{Lemma} \label{l.L.4.6}
The set $\mathcal E_A(H)$ is a $\pi$-core for $(K,D(K))$, and for  any $\varphi \in D( {K}\sostituzioneA)$ there exist $m\in \Nset$   and an  $m$-indexed sequence $(\varphi_{n_1,\ldots,n_m})\subset \mathcal E_A(H)$ such that
\begin{eqnarray} \label{e.L.23}
 \lim_{n_1\to\infty} \cdots \lim_{n_m\to\infty}  \frac{\varphi_{n_1,\ldots,n_m}}{1+|\cdot|^{}}
  &\stackrel{\pi}{=}& \frac{\varphi }{1+|\cdot|^{}}, \\
 \lim_{n_1\to\infty} \cdots \lim_{n_m\to\infty}  \frac{K_0\varphi_{n_1,\ldots,n_m}}{1+|\cdot|^{}}
  &\stackrel{\pi}{=}& \frac{ {K}\varphi }{1+|\cdot|^{}}. \label{e.L.23a}
\end{eqnarray} 
\end{Lemma}
\begin{proof}
{\bf Step 1}.
 Let\footnote{the assumpion $\lambda > \omega_0$ is necessary to define the resolvent of ${K}$ (cf Proposition \ref{p.L.2.9})}
 $\varphi \in D( {K}\sostituzioneA)$, $\lambda > \max\{0,\omega_0,\omega+M\kappa\}$ and set $f=\lambda\varphi- {K}\varphi$.
We fix a sequence $(f_{n_1})\subset C_b^1(H)$ such that  
\[
   \lim_{ n_1 \to\infty}\frac{f_{n_1}}{1+|\cdot|^{}}\stackrel{\pi}{=}\frac{f }{1+|\cdot|^{}}.
\]
Set $\varphi_{n_1}=R(\lambda, {K})f_{n_1}$.
By Proposition \ref{p.L.2.9} it follows
\begin{equation}   \label{e.L.24}
    \lim_{n_1\to\infty}   \frac{\varphi_{n_1} }{1+|\cdot|^{}}\stackrel{\pi}{=} \frac{\varphi }{1+|\cdot|^{}}, \quad \lim_{n_1\to\infty}\frac{K\varphi_{n_1} }{1+|\cdot|^{}}\stackrel{\pi}{=} \frac{ {K}\varphi }{1+|\cdot|^{}}.  
\end{equation}
{\bf Step 2}.
Now set $\varphi_{n_1,n_2}=R(\lambda,K_{n_2})f_{n_1}$,
where $K_{n_2}$ is the infinitesimal generator of the semigroup $P_t^{n_2}$, introduced in \eqref{e.L.17}.
Since $ f_{n_1}\in C_b^1(H)$, by Proposition \ref{P.L.7.2} we have $\varphi_{n_1,n_2}\in C_b^1(H)$ and 
\begin{equation} \label{e.L.25}
    \sup_{n_2\in \Nset} \|D\varphi_{n_1,n_2}\|_{C_b(H;H)}\leq \frac{M\|Df_{n_1}\|_{C_b(H;H)}}{\lambda-(\omega+M\kappa)},
\end{equation}
for any $ n_1\in \Nset$.
Moreover, by \eqref{e.L.19}  it holds
\begin{eqnarray}  
    \lim_{n_2\to\infty} \varphi_{n_1,n_2} \stackrel{\pi}{=}  \varphi_{n_1}, \quad
    \lim_{n_2\to\infty}K_{n_2}\varphi_{n_1,n_2}\stackrel{\pi}{=}K\varphi_{n_1}, \label{e.L.26}
\end{eqnarray} 
for any $n_1\in\Nset$.
Since $\varphi_{n_1,n_2}\in D(K_{n_2}\sostituzioneA)\cap C_b^1(H)$, by Theorem \ref{t.L.4.1} we have
\[
   K_{n_2}\varphi_{n_1,n_2}=L\varphi_{n_1,n_2}+\langle D\varphi_{n_1,n_2 },F_{n_2}\rangle.
\]
{\bf Step 3}.
By Proposition \ref{p.L.4.3}, for any $n_1,n_2$ there exists a sequence $ (\varphi_{n_1,n_2,n_3})\subset \mathcal E_A(H)$ (we still assume that it has only one index) such that
\begin{equation} \label{e.L.35}
  \lim_{n_3\to\infty}\varphi_{n_1,n_2,n_3}\stackrel{\pi}{=}\varphi_{n_1,n_2},\quad
  \lim_{n_3\to\infty}\frac{L\varphi_{n_1,n_2,n_3}}{1+|\cdot|}\stackrel{\pi}{=}\frac{L\varphi_{n_1,n_2}}{1+|\cdot|}
\end{equation}
and 
\[
     \lim_{n_3\to\infty}   \langle D\varphi_{n_1,n_2,n_3},h\rangle\stackrel{\pi}{=}
         \langle D\varphi_{n_1,n_2 },h\rangle.
\]
for any $h\in H$.
Notice that the since $F_{n_2}$ is globally Lipschitz, it follows
\[
     \lim_{n_3\to\infty} \frac{ \langle D\varphi_{n_1,n_2,n_3},F_{n_2}\rangle}{1+|\cdot|} \stackrel{\pi}{=}
        \frac{\langle D\varphi_{n_1,n_2 },F_{n_2}\rangle}{1+|\cdot|}.
\]
This, togheter with \eqref{e.L.35}, implies that the sequence  $ (\varphi_{n_1,n_2,n_3})$ fulfils
\[
  \lim_{n_3\to\infty}\varphi_{n_1,n_2,n_3}\stackrel{\pi}{=}\varphi_{n_1,n_2},\quad
  \lim_{n_3\to\infty}\frac{ K_{n_2}\varphi_{n_1,n_2,n_3}}{1+|\cdot|}\stackrel{\pi}{=}\frac{K_{n_2}\varphi_{n_1,n_2}}{1+|\cdot|}.
\]
Since ${K}$ is an extension of $K_0$ (cf Corollary \ref{c.L.4.6}), we have
\[
     {K}\varphi_{n_1,n_2,n_3}=K_0\varphi_{n_1,n_2,n_3}=K_{n_2}\varphi_{n_1,n_2,n_3}+
    \langle D\varphi_{n_1,n_2,n_3},F-F_{n_2}\rangle
\]
for any $n_1,n_2,n_3 \in \Nset$.
So we find
\begin{equation}\label{e.L.27}
  \lim_{n_3\to\infty}  \frac{K_0\varphi_{n_1,n_2,n_3}}{1+|\cdot|^{}}\stackrel{\pi}{=}\frac{K_{n_2}\varphi_{n_1,n_2 }+
    \langle D\varphi_{n_1,n_2 },F-F_{n_2}\rangle}{1+|\cdot|^{}},
\end{equation}
for any $ n_1,n_2 \in \Nset$.
Moreover, by \eqref{e.L.25}, we see that 
\[
     \frac{|\langle D\varphi_{n_1,n_2 }(x),F(x)-F_{n_2}(x)\rangle|}{1+|x|^{}}\leq \frac{M\|Df_{n_1}\|_{C_b(H;H)}}{\lambda-(\omega+M\kappa)}\frac{|F(x)-F_{n_2}(x)|}{1+|x|^{}} 
\]
and consequently 
\[
    \lim_{n_2\to\infty}   \frac{ \langle D\varphi_{n_1,n_2 },F-F_{n_2}\rangle}{1+|\cdot|^{}}\stackrel{\pi}{=}0.
\]
This, togheter with \eqref{e.L.26} implies
\begin{equation}\label{e.L.28a}
   \lim_{n_2\to\infty}\frac{K_{n_2}\varphi_{n_1,n_2 }+
    \langle D\varphi_{n_1,n_2 },F-F_{n_2}\rangle}{1+|\cdot|^{}}\stackrel{\pi}{=}\frac{K\varphi_{n_1}}{1+|\cdot|^{}}.
\end{equation}
Finally, by taking into account \eqref{e.L.24}, \eqref{e.L.27}; \eqref{e.L.28a}, the sequence $(\varphi_{n_1,n_2,n_3})\subset \mathcal E_A(H) $ fulfils 
\[
    \lim_{n_1\to\infty}\lim_{n_2\to\infty}\lim_{n_3\to\infty} \frac{\varphi_{n_1,n_2,n_3}}{1+|\cdot|^{}}
      \stackrel{\pi}{=}\frac{\varphi }{1+|\cdot|^{}},
\]
\[
   \lim_{n_1\to\infty}\lim_{n_2\to\infty}\lim_{n_3\to\infty}\frac{K_0\varphi_{n_1,n_2,n_3}}{1+|\cdot|^{}}
      \stackrel{\pi}{=}\frac{ {K}\varphi }{1+|\cdot|^{}}.
\]
This concludes the proof. 
\end{proof}
\section{Proof of Theorem \ref{t.L.1.4}} \label{s.proof3}
Let $\varphi\in D( {K})$ and assume that 
$(\varphi_n)_{n\in\Nset} \subset \mathcal E_A(H)$ fulfils 
\eqref{e.L.23}, \eqref{e.L.23a} (for simplicity the assume that this sequence has only one index: this does not change the generality of the proof).
For any $t \geq0$ we find
\begin{eqnarray*}
   \int_H\varphi(x)\mu_t(dx)-\int_H\varphi(x)\mu (dx)&=&\lim_{n\to\infty}
   \left( \int_H\varphi_n(x)\mu_t(dx)-\int_H\varphi_n(x)\mu (dx)\right)\\
  &=&  \lim_{n\to\infty}\int_0^t\left( \int_HK_0\varphi_n(x)\mu_s(dx)\right)ds,
\end{eqnarray*}
since $\varphi_n\in D( {K})$ and ${K}\varphi_n=K_0\varphi_n$, for any $n\in \Nset$ (cf Corollary \ref{c.L.4.6}).
Now observe that since $\sup_{n\in\Nset}|K_0\varphi_n(x)|\leq c(1+|x|^{}) $ for some $c>0$ and since $\mu_s\in \mathcal M_1(H)$
for any $s\geq0$, it holds
$$
\lim_{n\to\infty}\int_HK_0\varphi_n(x)\mu_s(dx)=\int_H  {K}\varphi(x)\mu_s(dx)
$$
and 
$$
\sup_{n\in\Nset}\left|\int_HK_0\varphi_n(x)\mu_s(dx)\right|\leq c\int_H(1+|x|^{})|\mu_s|(dx).
$$
Hence, by taking into account \eqref{e.L.5b} 
we can apply the dominated convergence theorem to obtain 
$$
    \lim_{n\to\infty}\int_0^t\left( \int_HK_0\varphi_n(x)\mu_s(dx)\right)ds=
   \int_0^t\left( \int_H  {K}\varphi(x)\mu_s(dx)\right)ds
$$
So, $\{\mu_t\}_{t\geq0}$ is also a solution of the measure equation for $(  {K},D( {K}\sostituzioneA))$.
Since by Theorem \ref{t.L.1} such a solution is unique and it is given by $\{P_t^*\mu\}_{t\geq0}$, 
we have $\int_H\varphi(x)P_t^*\mu(dx)=\int_H\varphi(x)\mu_t(dx)$, for any $\varphi\in \mathcal E_A(H)$.
By taking into account that the set $\mathcal E_A(H)$ is $\pi$-dense in $C_b(H)$ (cf. Proposition \ref{p.L.6.2}), we have
 $\int_H\varphi(x)P_t^*\mu(dx)=\int_H\varphi(x)\mu_t(dx)$, for any $\varphi\in   C_b(H)$.
this clearly implies $P_t^*\mu=\mu_t$, $\forall t\geq0$.
This concludes the proof.

\section{The reaction-diffusion case} \label{s.application}

We shall consider here the stochastic heat equation perturbed by a polynomial term  of odd degree $d>1$ having
negative leading coefficient (this will ensures non explosion). 
We shall  represent this 
 polynomial as
$$
\lambda \xi-p(\xi),\quad \xi\in \Rset,
$$
where $\lambda\in \Rset$ and $p$ is an increasing polynomial, that is    $p'(\xi)\geq 0$ for all $\xi\in \Rset.$

We   set $H=L^2(\mathcal O)$ where $\mathcal O=[0,1]^n$, $n\in \Nset$, and denote by  $\partial \mathcal O$ the
boundary of $\mathcal O$.  
We are concerned with the following stochastic differential  equation with Dirichlet boundary conditions
\begin{equation} \label{e.RD.4.1}
\left\{\begin{array}{lll}
dX(t,\xi)=[\Delta_\xi X(t,\xi)+\lambda X(t,\xi) -p(X(t,\xi))]dt+BdW(t,\xi),\quad \xi\in \mathcal O,\\
\\
X(t,\xi)=0,\quad t\ge 0,\; \xi\in \partial \mathcal O,\\
\\
X(0,\xi)=x(\xi),\quad \xi\in \mathcal O,\;x\in H,
\end{array}\right.
\end{equation}
where $\Delta_\xi$ is the Laplace operator,  $B\in \mathcal L(H)$ and  $W$  is a cylindrical Wiener process in a 
stochastic basis
$(\Omega,\mathcal F, (\mathcal F_t)_{t\geq0}, \PP)$ in $H.$  We choose  $W$ of the form
$$
W(t,\xi)=\sum_{k=1}^\infty e_k(\xi)\beta_k(t),\quad \xi\in \mathcal O,\;t\ge 0,
$$
where $\{e_k\}$ is a complete orthonormal system in $H$ and  $\{\beta_k\}$    a sequence  of mutually 
independent
standard Brownian motions on   $(\Omega,\mathcal F, (\mathcal F_t)_{t\geq0}, \PP)$.

Let us write problem \eqref{e.RD.4.1} as a stochastic differential equation
  in the Hilbert space $H$.
 For this we denote by $A$ the realization of the Laplace operator with
Dirichlet boundary conditions,
\begin{equation}   \label{e.RD.4.3}
  \left\{\begin{array}{lll}
          Ax=\Delta_\xi x,\quad x\in D(A),\\
           \\
    \DS {D(A)=H^2(\mathcal O)\cap H^1_0(\mathcal O).}
        \end{array}\right.
\end{equation}
  $A$ is self--adjoint and has a complete orthonormal system of eigenfunctions, namely
$$
   e_k(\xi)=(2/\pi)^{n/2}\;\sin (\pi k_1\xi_1)\cdots (\sin \pi k_n\xi_n),
$$
where $k=(k_1,...,k_n),$ $k_i\in \Nset$.
For any $x\in H$ we   set $x_k= \langle x,e_k  \rangle,\, k\in \Nset^n. $
Notice that
$$
    A e_k=-\pi^2|k|^2,\quad k\in \Nset^n,\;  |k|^2=k^2_1+\cdots+k^2_n. 
$$
 Therefore, we have
\begin{equation}  \label{e.RD.4.4}
     \|e^{tA}\|\le e^{-\pi^2t},\quad t\ge 0.
\end{equation}
\begin{Remark}   \label{r4.1}
{\em We can also consider the realization of the Laplace operators with Neumann boundary conditions
\[
 \left\{\begin{array}{lll}
    Nx=\Delta_\xi x,\quad x\in D(N),\\
      \\
   \DS{D(N)=\left\{x\in H^2(\mathcal O):\;\frac{\partial x}{\partial \eta}\;=0\;\mbox{\rm on}\;\partial\mathcal O\right\}}
     \end{array}\right.
\]
where $\eta$ represents the outward normal to $\partial\mathcal O$.
Then
$$
N f_k=-\pi^2|k|^2 f_k,\quad k\in (\Nset\cap\{0\})^n,
$$
where
$$
f_k(\xi)=(2/\pi)^{n/2}\;\cos (\pi k_1\xi_1)\cdots (\cos \pi k_n\xi_n),
$$
  $k=(k_1,...,k_n),$ $k_i\in \Nset\cup\{0\}$ and $|k|^2=k^2_1+\cdots+k^2_n$.
}
\end{Remark}

Concerning the operator $B$ we shall assume, for the sake of simplicity (\footnote{
All following results remain true taking $B=G(-A)^{-\gamma/2}$ with $G\in\mathcal L(H)$.}), that $B=(-A)^{-\gamma/2}$,    
 where
\begin{equation}   \label{e.RD.4.11}
      \gamma>\frac n2 -1.
\end{equation}

Now, setting $X(t)=X(t,\cdot)$ and $W(t)=W(t,\cdot),$ we shall write problem \eqref{e.RD.4.1}
as
\begin{equation}
\label{e.RD.4.41}
\left\{\begin{array}{lll}
dX(t)=[AX(t)+F(X(t))]dt+(-A)^{-\gamma/2}dW(t),\\
\\
X(0)=x.
\end{array}\right.
\end{equation}
where $F$ is the mapping
$$
F:D(F)=L^{2d}(\mathcal O)\subset H\to H,\;x(\xi)\mapsto \lambda \xi-p(x(\xi)).
$$
It is well known that for any $x\in L^{2d}(\mathcal O)$ problem \eqref{e.RD.4.41} has a  unique mild solution  $(X(t,x))_{t\geq0,x\in H}$ (see, for instance, \cite{Cerrai01}, \cite{DP04}), fulfilling
\begin{equation} \label{e.RD.6}
 X(t,x)= e^{tA}x+\int_0^te^{(t-s)A}BdW(s)+\int_0^te^{(t-s)A}F(X(s,x))ds
\end{equation}
for any $t\geq0$.
Finally, it  is well known that for any $T>0$ there exists $c>0$ such that
\begin{equation}   \label{e.RD.6a}
  \sup_{t\in[0,T]}  \EE\left[ |X(t,x)|_{L^{2d}(\mathcal O)}^{d} \right] \leq c 
         \left(1+|x|_{L^{2d}(\mathcal O)}^{d}\right).
\end{equation}
\begin{equation}   \label{e.RD.6b}
|X(t,x)-X(t,y)|\leq e^{(\lambda-\pi^2)t}|x-y|,
\end{equation}
see \cite[Theorem 4.8]{DP04}.
\subsection{Main results}
We consider here the  Kolmogorov operator
\begin{equation} \label{e.RD.4}
 K_0\varphi(x)= \frac12\textrm{Tr}\big[BB^*D^2\varphi(x)\big]+\langle x,A^*D\varphi(x)\rangle+\langle D\varphi(x),F(x)\rangle,\,x\in L^{2d}(\mathcal O).
\end{equation}
We are interested in extending the results of Theorems \ref{t.L.1}, \ref{t.L.2}, \ref{t.L.1.4} to this operator. 
This will be done in Theorems \ref{t.RD.1}, \ref{t.RD.2}, \ref{t.RD.1.4} respectively.

Denote by $C_{b,d}(L^{2d}(\mathcal O))$ the space of all   functions $\varphi:L^{2d}(\mathcal O)\to \Rset$ such that the function
\[
 L^{2d}(\mathcal O)\to\Rset,\quad x\to \frac{\varphi(x)}{1+|x|_{L^{2d}(\mathcal O)}^d} 
\]
is uniformly continuous and bounded. 
The space $C_{b,d}(L^{2d}(\mathcal O))$, endowed with the norm 
\[
   \|\varphi\|_{C_{b,d}(L^{2d}(\mathcal O))}=\sup_{x\in L^{2d}(\mathcal O)} \frac{|\varphi(x)|}{1+|x|_{L^{2d}(\mathcal O)}^{d}} 
\]
is a Banach space.
Thanks to estimates \eqref{e.RD.6a} and \eqref{e.RD.6b} we can define a semigroup of transitional operators in $C_{b,d}(L^{2d}(\mathcal O))$, by the formula
\begin{equation} \label{e.RD.7}
    {P}_t\varphi(x)= \EE\bigl[\varphi( X(t,x))   \bigr],\quad t\geq0,\,\varphi\in C_{b,d}(L^{2d}(\mathcal O)),\, x\in L^{2d}(\mathcal O),
\end{equation}
see Proposition \ref{p.RD.3.1}.
We define its infinitesimal generator by setting
\begin{equation}  \label{e.RD.0}
\begin{cases}
\DS D( {K})=\bigg\{ \varphi \in C_{b,d}(L^{2d}(\mathcal O)): \exists g\in C_{b,d}(L^{2d}(\mathcal O)),\, \lim_{t\to 0^+} \frac{ {P}_t\varphi(x)-\varphi(x)}{t}=  \\ 
  \DS  \qquad\qquad\quad =g(x),\,x\in L^{2d}(\mathcal O),\;\sup_{t\in(0,1)}\left\|\frac{ {P}_t\varphi-\varphi}{t}\right\|_{C_{b,d}(L^{2d}(\mathcal O))}<\infty \bigg\}\\   
   {}   \\
  \DS  {K}\varphi(x)=\lim_{t\to 0^+} \frac{ {P}_t\varphi(x)-\varphi(x)}{t},\quad \varphi\in D( {K}),\,x\in L^{2d}(\mathcal O).
\end{cases}
\end{equation}
Let  $\mathcal M_{d}(L^{2d}(\mathcal O))$ be the space of all Borel finite measures on $L^{2d}(\mathcal O)$ such that
\[
 \int_{L^{2d}(\mathcal O)}|x|_{L^{2d}(\mathcal O)}^{d}|\mu|(dx)<\infty.
\]
Since $L^{2d}(\mathcal O)\subset H$, we have $\mathcal M_{d}(L^{2d}(\mathcal O))\subset \mathcal M(H)$.
The following theorem generalize Theorem \ref{t.L.1} to the reaction-diffusion case.
\begin{Theorem} \label{t.RD.1}
Let $( {P}_t)_{t\geq0}$ be the semigroup defined by \eqref{e.RD.7} 
 $C_{b,2}(H)$, 
and   let  $( {K},D( {K} ))$ be its infinitesimal generator in $C_{b,d}(L^{2d}(\mathcal O))$, defined   by  \eqref{e.RD.0}.
Then,  the formula 
$$
  \langle \varphi,  {P}_t^*F\rangle_{\mathcal L(C_{b,d}(L^{2d}(\mathcal O)),\,(C_{b,d}(L^{2d}(\mathcal O)))^*)} = \langle  {P}_t\varphi, F\rangle_{\mathcal L(C_{b,d}(L^{2d}(\mathcal O)),\, (C_{b,d}(L^{2d}(\mathcal O)))^*)}
$$
defines a semigroup $( {P}_t^*)_{t\geq0}$ of linear and continuous operators  on $(C_{b,d}(L^{2d}(\mathcal O)))^*$ that maps $\mathcal M_{d}(L^{2d}(\mathcal O))$ into $\mathcal M_{d}(L^{2d}(\mathcal O))$.
Moreover,  for any $\mu\in \mathcal M_{d}(L^{2d}(\mathcal O))$ there exists a unique family of measures  $\{\mu_t\}_{t\geq0}\subset \mathcal M_{d}(L^{2d}(\mathcal O))$
such that
\begin{equation} \label{e.RD.5b}
    \int_0^T\left(\int_H|x|_{L^{2d}(\mathcal O)}^d|\mu_t|(dx)\right)dt<\infty,\quad \forall T>0 
\end{equation}
and   
\begin{equation}   \label{e.RD.8} 
   \int_H\varphi(x)\mu_t(dx)- \int_H\varphi(x)\mu(dx)=
  \int_0^t\left( \int_H {K}\varphi(x)\mu_s(dx)   \right)ds,
\end{equation}
 $t\geq0$, $\varphi\in D({K} )$.
Finally, the solution of \eqref{e.RD.5b}, \eqref{e.RD.8}  is given by $\{{P}_t^*\mu\}_{t\geq0}$.  
\end{Theorem}
It worth to note that $C_b(H)\subset C_{b,1}(H)\subset C_{b,d}(L^{2d}(\mathcal O))$, with continuous embedding.
This argument will be used in what follows.
Note, also, that for any $\varphi\in C_{b,d}(L^{2d}(\mathcal O))$ there exists a sequence $(\varphi_n)_{n\in \Nset}\subset C_b(H)$
such that
\[
 \lim_{n\to\infty}\frac{\varphi_n}{1+|\cdot|_{L^{2d}(\mathcal O)}^{d}} \stackrel{\pi}{=}\frac{\varphi }{1+|\cdot|_{L^{2d}(\mathcal O)}^{d}}.
\]
In the above formula we have to understand $(1+|x|_{L^{2d} (\mathcal O)}^d )^{-1}\varphi(x)=0$ if $x\in H\setminus L^{2d} (\mathcal O)$.
This allow us to use the $\pi$-convergence also for functions belonging to the space $C_{b,d}(L^{2d}(\mathcal O))$.
We denote by $\mathcal E_A(H)$ the linear span of the real and imaginary parts of the functions\footnote{Here $A$ is self-adjoint, hence we take $h\in D(A)$ (cf section 1).}
\[
  H\to \Cset,\quad x\mapsto e^{i\langle x, h\rangle},\quad h\in D(A).
\]
The main result of this section is the following
\begin{Theorem} \label{t.RD.2} 
The infinitesimal operator  $( {K},D( {K}))$ defined in \eqref{e.RD.0} is an extension of  $K_0$, and for any $\varphi\in \mathcal E_A(H)$    we have  $\varphi\in D(K)$ and ${K}\varphi=K_0\varphi$.
Moreover, the set $ \mathcal E_A(H)$ is a $\pi$-core for $( {K},D( {K}))$, that is for any $\varphi\in D( {K})$
there exist  $m\in \Nset$ and an $m$-indexed sequence $(\varphi_{n_1,\ldots,n_m})_{n_1\in\Nset,\ldots,n_m\in\Nset}\subset \mathcal E_A(H)$ such that
\begin{equation} \label{e.RD.48}
   \lim_{n_1\to \infty}\cdots\lim_{n_m\to \infty} \frac{\varphi_{n_1,\ldots,n_m}}{1+|\cdot|_{L^{2d}(\mathcal O)}^{d}}\stackrel{\pi}{=}\frac{\varphi}{1+|\cdot|_{L^{2d}(\mathcal O)}^{d}}
\end{equation}
and
 \begin{equation}\label{e.RD.49}
   \lim_{n_1\to \infty}\cdots\lim_{n_m\to \infty} \frac{K_0\varphi_{n_1,\ldots,n_m}}{1+|\cdot|_{L^{2d}(\mathcal O)}^{d}}\stackrel{\pi}{=}\frac{ {K}\varphi}{1+|\cdot|_{L^{2d}(\mathcal O)}^{d}}.
 \end{equation}
\end{Theorem}
Thanks to  Theorem \ref{t.RD.2} we are able to prove the following
\begin{Theorem} \label{t.RD.1.4}
For any $\mu \in \mathcal M_{d}(L^{2d}(\mathcal O))$  there exists an unique family of measures $\{\mu_t\}_{t\geq0}\subset \mathcal M_{d}(L^{2d}(\mathcal O)) $ fulfilling \eqref{e.RD.5b}
 and the measure equation 
\begin{equation}\label{e.RD.2.5a}
\int_H\varphi(x)\mu_t(dx)- \int_H\varphi(x)\mu(dx)=
  \int_0^t\left( \int_HK_0\varphi(x)\mu_s(dx)   \right)ds, 
\end{equation}
$t\geq0,\,\varphi \in\mathcal E_A(H)$.
Finally,   the solution of \eqref{e.RD.5b}, \eqref{e.RD.2.5a} is given by $  \{{P}_t^*\mu\}_{t\geq0}$.
\end{Theorem}
In the next section we  study the transition semigroup  \eqref{e.RD.7} and its infinitesimal generator \eqref{e.RD.0} in the space $C_{b,d}(L^{2d}(\mathcal O)) $.
In section \ref{s.RD.OU}  we shall introduce an approximation of problem \eqref{e.RD.4.41} that will be often used in what follows.
Finally, in sections \ref{s.RD.6.4}, \ref{s.RD.6.5}, \ref{s.RD.6.6} we prove Theorems \ref{t.RD.1}, \ref{t.RD.2}, \ref{t.RD.1.4}, respectively.

\subsection{The transition semigroup in $C_{b,d}(L^{2d}(\mathcal O))$} \label{s.RD.6.3}
 The  following two propositions may be proved in much the same way as Proposition \ref{p.L.2.1} and Proposition \ref{p.L.2.9}.
\begin{Proposition} \label{p.RD.3.1}
 Formula \eqref{e.RD.7} semigroup of operators $({P}_t)_{t\geq0} $ in $C_{b,d}(L^{2d}(\mathcal O)) $,
and there exists  a family of probability measures $\{\pi_t(x,\cdot),\,t\geq0,\,x\in L^{2d}(\mathcal O)\}\subset  \mathcal M_d(L^{2d}(\mathcal O))$ 
and two constants $c_0, \omega_0>0$, 
such that
\begin{itemize}
\item[(i)] $ {P}_t\in \mathcal L(C_{b,d}(L^{2d}(\mathcal O)) )$ and  $  \|  {P}_t\|_{\mathcal L(C_{b,d}(L^{2d}(\mathcal O))  )} \leq c_0 e^{ \omega_0t}$; 
\item[(ii)] $\DS  {P}_t\varphi(x)=\int_H\varphi(y)\pi_t(x,dy) $, for any $t\geq0$, $\varphi \in C_{b,d}(L^{2d}(\mathcal O))$, 
 $x\in L^{2d}(\mathcal O)$;
\item[(iii)] for any  $\varphi \in C_{b,d}(L^{2d}(\mathcal O))$, $x\in H$, the function $\Rset^+\to\Rset$, $t\mapsto  {P}_t\varphi(x)$ is continuous. 
\item[(iv)] $  {P}_t  {P}_s=  {P}_{t+s}$, for any $t,s\geq0$ and $  {P}_0=I$;
\item[(v)] for any $\varphi\in C_{b, d}(L^{2d}(\mathcal O))$ and any sequence $(\varphi_n)_{n\in \Nset}\subset C_{b, d}(L^{2d}(\mathcal O))$ such that
\[
  \lim_{n\to\infty} \frac{\varphi_n}{1+|\cdot|_{L^{2d}(\mathcal O)}^d}  
     \stackrel{\pi}{=}\frac{\varphi }{1+|\cdot|_{L^{2d}(\mathcal O)}^d} 
\]
 we have, for any $t\geq0$,
\[
  \lim_{n\to\infty} \frac{{P}_t\varphi_n}{1+|\cdot|_{L^{2d}(\mathcal O)}^d}  
    \stackrel{\pi}{=}\frac{{P}_t\varphi }{1+|\cdot|_{L^{2d}(\mathcal O)}^d}. 
\]  
\end{itemize}
\end{Proposition}
 
\begin{Proposition} \label{p.RD.6.0}
Under the hypothesis of Proposition \ref{p.RD.3.1},
let $(K,D(K\sostituzioneA)) $ 
be  the infinitesimal generator     \eqref{e.RD.0}. 
Then
\begin{itemize} 
\item[(i)] for any $\varphi \in D( {K})$, we have $  {P}_t\varphi \in D( {K}\sostituzioneA))$ and 
     $  {K} {P}_t\varphi =  {P}_t  {K}\varphi$, $t\geq0$;
\item[(ii)] for any $\varphi \in D( {K}\sostituzioneA)$, $x\in L^{2d}(\mathcal O)$, the map $[0,\infty)\to \Rset$, $t\mapsto  {P}_t\varphi(x)$ is continuously differentiable and $(d/dt) {P}_t\varphi(x) =  {P}_t  {K}\varphi(x)$; 
\item[(iii)]   given $c_0, \omega_0>0$  as in Proposition \ref{p.RD.3.1},  for any $\omega>\omega_0$ the linear operator $R(\omega, {K})$ on $C_{b,d}(L^{2d}(\mathcal O))$ done by
$$
   R(\omega, {K})f(x)=\int_0^\infty e^{-\omega t} {P}_tf(x)dt, \quad f\in C_{b,d}(L^{2d}(\mathcal O)),\,x\in L^{2d}(\mathcal O)
$$
satisfies, for any  $f\in C_{b,1}(H)$
$$
   R(\omega, {K})\in \mathcal L(C_{b,d}(L^{2d}(\mathcal O))),\quad \quad \|R(\omega, {K})\|_{\mathcal L(C_{b,d}(L^{2d}(\mathcal O)))}\leq 
\frac{c_0 }{\omega-\omega_0} 
$$
$$
  R(\omega, {K})f\in D( {K}\sostituzioneA),\quad (\omega I- {K})R(\omega, {K})f=f.
$$
We call $R(\omega, {K})$ the {\em resolvent} of $  {K}$ at $\omega$.
\end{itemize}
\end{Proposition}

\subsection{Some auxiliary results  } \label{s.RD.OU}

 It is convenient to consider the Ornstein--Uhlenbeck process
$$
\left\{\begin{array}{lll}
dZ(t)=AZ(t)dt+(-A)^{-\gamma/2}dW(t),\\
\\
Z(0)=x,
\end{array}\right.
$$
and the corresponding transition semigroup in $C_{b,1}(H)$
\begin{equation} \label{e.RD.50}
  R_t\varphi(x)=\EE[\varphi(Z(t,x))],\quad\varphi\in C_{b,1}(H).
\end{equation}
Notice that thanks to \eqref{e.RD.4.4}, \eqref{e.RD.4.11} the operator 
$$
\begin{array}{lll}
Q_tx &=& \DS{\int_0^te^{sA}BB^*e^{sA*}xds=\int_0^t(-A)^{-\gamma}e^{2tA}xdt}\\
\\
&=&\frac12(-A)^{-(1+\gamma)}(1-e^{2tA})x,\quad t\ge
0, x\in H,
\end{array}
$$
is of trace class.
This implies that the Ornstein-Uhlenbeck process $Z(t,x)$ has gaussian law of mean $e^{tA}x$ and covariance operator $Q_t$, and the representation formula
\[
      R_t\varphi(x)=\int_H\varphi(e^{tA}x+y)\mathcal N_{Q_t}(dy)
\]
holds for any $t\geq0$, $\varphi\in C_{b,1}(H)$, $x\in H $.
Notice that we can take $\gamma=0$    and $B=I$ (white noise)  only for $n=1$.
As in section \ref{s.L.OU}, we denote by $(L,D(L))$ the infinitesimal generator of the Ornstein-Uhlenbeck semigroup $(R_t)_{t\geq0}$ in the space $C_{b,1}(H)$.

A basic tool we use to prove our results  is provided by the following approximating problem
\begin{equation}
\label{e.RD.4.16}
\left\{\begin{array}{l}
dX^n (t)=(AX^n (t)+F_n (X^n (t))dt+(-A)^{-\gamma/2}dW(t),\\
\\
X^n (0)=x\in H,
\end{array}\right.
\end{equation}
where for any $n\in\Nset,$   $F_n$  is defined by
$$
F_n(x)(\xi)=\lambda x(\xi)-p_n (x(\xi)),
$$
and $p_n$ is defined by
\[
  p_n(\eta)=\frac{np(\eta)}{\sqrt{n^2+p^2(\eta)}},\quad \eta\in \Rset.
\]
Notice that  $p_n$ is bounded and differentiable, with bounded derivative 
\[
  p_n'(\eta)=\frac{np'(\eta)}{\sqrt{n^2+p^2(\eta)}}\left(1-  \frac{p^2(\eta)}{n^2+p^2(\eta)} \right)\geq 0,  
\]
for any $n\in \Nset$, $ \eta\in \Rset$.
Clearly, $|p_n(\eta)|\leq|p (\eta)|$, $\eta\in \Rset $ and $p_n(\eta)\to p(\eta)$ as $n\to \infty$, for any $\eta\in \Rset $.
$F_n$ is Lipschitz continuous, and for any $n\in \Nset$, $x\in H$ problem \eqref{e.RD.4.16} has a unique mild solution $(X^n(t,x))_{t\geq0}$ (cf section 1).
Since by the above discussion we have $|F_n(x)|\leq|F (x)|$, $x\in H $ and $|F_n(x)| \to |F (x)|$ as $n\to \infty$, for any $x\in H $
it is not difficult but tedious to show that for any $x\in L^{2d}(\mathcal O)$ it holds
\begin{equation} \label{e.RD.52}
  \lim_{n\to\infty}\sup_{t\in[0,T]}\EE\left[|X^{ n }(t,x) - X(t,x)|^2\right]=0
\end{equation}
and 
\begin{equation} \label{e.RD.52a}
  \EE\left[|X^{ n }(t,x)|_{L^{2d}(\mathcal O)}^d\right]\leq \EE\left[|X (t,x)|_{L^{2d}(\mathcal O)}^d\right], \quad n\in\Nset.
\end{equation}
\begin{Proposition} \label{p.RD.6.7}
 For any $n\in \Nset$, let $(P_t^n)_{t\geq0}$ be the transitional semigroup associated to the mild solution of problem \eqref{e.RD.4.16} in the space $C_{b,d}(L^{2d}(\mathcal O))$, defined as in \eqref{e.RD.7} with $X^n(t,x)$ replacing $X(t,x)$.
Then 
\begin{itemize}
 \item[{\em (i)}]  $(P_t^n)_{t\geq0}$ satisfies statements (i)--(v) of Proposition  \ref{p.RD.3.1}, and for $c_0,\omega_0$ as in Proposition \ref{p.RD.3.1} we have
$\DS
  \|P_t^n\|_{\mathcal L(C_{b,d}(L^{2d}(\mathcal O)))}\leq c_0e^{\omega_0t};
$
\item[{\em (ii)}]   $(P_t^n)_{t\geq0}$ is a semigroup of operators in the space $C_{b,1}(H)$, and it satisfies statements (i)--(v) of Proposition \ref{p.L.2.1}.
In particular, there exists $c_n, \omega_n>0$ such that 
$\DS
  \|P_t^n\|_{\mathcal L(C_{b,1}(H))}\leq c_ne^{\omega_nt}$, for any $t\geq0$.
\end{itemize}
\end{Proposition}
\begin{proof}
(i) follows by \eqref{e.RD.52a}.
(ii) follows since equation \eqref{e.RD.4.16} satisfies Hypothesis \ref{h.L.3.1}.
\end{proof}

By (ii) of Proposition \ref{p.RD.6.7}, we can define, for any $n\in \Nset$, the infinitesimal generator $(K_n,D(K_n))$ of the semigroup  $(P_t^n)_{t\geq0}$ in the space $C_{b,1}(H)$ (cf \eqref{e.L.0}). 

By Theorem \ref{t.L.4.1} and Proposition \ref{p.RD.6.7} it follows
\begin{Proposition} \label{p.RD.6.8}
 For any $n\in \Nset$ we have $D(L)\cap C_b^1(H)=D(K_n)\cap C_b^1(H)$, and for any $\varphi\in D(L)\cap C_b^1(H)$ we have
$K_n\varphi=L\varphi+\langle D\varphi,F_n\rangle$.
\end{Proposition}
The semigroup $(P_t^n)_{t\geq0}$ enjoyes the following property, which will be essential in the proof of Theorem \ref{t.RD.2}.
\begin{Proposition} \label{p.RD.9}
 For any $n\in \Nset$, the semigroup $(P_t^n)_{t\geq0}$ maps $C_b^1(H)$ into $C_b^1(H)$, and for any $\varphi\in C_b^1(H)$ it holds
\[
   |DP_t\varphi(x)| \leq e^{2(\lambda-\pi^2)t}\sup_{x\in H}|D\varphi(x)|
\]
\end{Proposition}

\begin{proof}
 Since the nonlinearity $F_n$ is differentiable with uniformly continuous and bounded differential, it is well known (see, for instance, \cite{DPZ92}) that the mild solution $X^n(t,x)$ of problem \eqref{e.RD.4.16}  is differentiable with respect to  $x$ and for any $x,h\in H$  
we have $DX^n(t,x)\cdot h=\eta_n^h(t,x)$, where $\eta_n^h(t,x)  $ is the mild solution of the differential equation with random coefficients 
\[
  \begin{cases}
   \DS \frac{d}{dt} \eta_n^h(t,x)= A \eta_n^h(t,x)+DF_n(X^n(t,x))\cdot  \eta_n^h(t,x)&t\geq0 \\
      \eta_n^h(t,x)=0.
  \end{cases}
\]
By multiplying by $\eta_n^h(t,x)$ and by integrating over $\mathcal O$ we find
\[
 \frac12\frac{d}{dt} |\eta_n^h(t,x)|^2= \langle (A+\lambda) \eta_n^h(t,x),\eta_n^h(t,x) \rangle 
 -\int_{\mathcal O} p'_n(X^n(t,x)(\xi) ) |\eta_n^h(t,x)(\xi)|^2d\xi. 
\]
By taking into account that $p'_n\geq0$ and by integrating by parts we find 
\[
  \frac12\frac{d}{dt} |\eta_n^h(t,x)|^2+\int_{\mathcal O}|D_\xi \eta_n^h(t,x)(\xi)|^2d\xi \leq \lambda  |\eta_n^h(t,x)|^2.
\]
Now,  the classical Poincar\'e inequality implies $|D_\xi \eta_n^h(t,x)|\geq \pi^2|  \eta_n^h(t,x)|$
and we obtain
\[
  \frac12\frac{d}{dt} |\eta_n^h(t,x)|^2\leq (\lambda-\pi^2) |\eta_n^h(t,x)|^2, \quad x\in H, t\geq0.
\]
Consequently,   by the Gronwall lemma we find
 \begin{equation} \label{e.RD.57a}
  |\eta^h(t,x)|\leq e^{2(\lambda-\pi^2)t}|h|.
\end{equation}
Now take $\varphi\in C_b^1(H)$.
For any $x,h\in H$ we have
\[
  DP_t^n\varphi(x)\cdot h=\EE\left[D\varphi(X^n(t,x))\cdot \eta^h(t,x)   \right].
\]
Hence by \eqref{e.RD.57a} 
\[
  |DP_t^n\varphi(x)\cdot h|\leq\EE\left[|D\varphi(X^n(t,x))|| \eta^h(t,x)|   \right]\leq \sup_{x\in H}|D\varphi(x)|e^{2(\lambda-\pi^2)t}|h|,
\]
which implies the result.
\end{proof}

\subsection{Proof of Theorem \ref{t.RD.1}} \label{s.RD.6.4}

We have first to show that $(P_t^*)_{t\geq0}$ is a semigroup of linear and continuous  operators in $  (C_{b,d}(L^{2d}(\mathcal O)))^*   $ and that $P_t^*\mu\in \mathcal M_d(L^{2d}(\mathcal O))$ for any $t\geq0$, $\mu\in \mathcal M_d(L^{2d}(\mathcal O)) $.
These facts follow  by Proposition \ref{p.RD.3.1} and by the argument of Lemma \ref{l.L.3}.
We left the details to the reader.

We now show existence of a solution for the measure equation, namely we show that $\{P_t^*\mu\}_{t\geq0}$ fulfils \eqref{e.RD.2.5a}, \eqref{e.RD.5b}.
To show that $\{P_t^*\mu\}_{t\geq0}$ fulfils  \eqref{e.RD.2.5a} it can be used the argument in Lemma \ref{l.L.5.3}.
We left the details to the reader.
We now check  that \eqref{e.RD.5b} holds. 
Fix $T>0$.
By the local boundedness of the operators $P_t^*\mu $ and by the semigroup property it  follows that there exists $c>0$ such that 
\[
    \sup_{t\in[0,T]}\|P_t^* \|_{\mathcal L((C_{b,d}(L^{2d}(\mathcal O)))^*)}\leq c.
\]
Still by the first part of the theorem, since $\mu\in \mathcal M_d(L^{2d}(\mathcal O) )$ we have $P_t^*\mu\in \mathcal M_d(L^{2d}(\mathcal O))$. 
Hence
\begin{eqnarray*}
&& \int_0^T\left(\int_H |x|_{L^{2d}(\mathcal O)}^d|P_t^*\mu|(dx)\right)dt =
 \int_0^T\left(\int_{L^{2d}(\mathcal O)} |x|_{L^{2d}(\mathcal O)}^d|P_t^*\mu|(dx)\right)dt 
\\
 &&\qquad  \leq   \int_0^T\|P_t^*\mu\|_{(C_{b,d}(L^{2d}(\mathcal O)))^*}dt  
    \leq c\int_0^T\| \mu\|_{(C_{b,d}(L^{2d}(\mathcal O)))^*}dt
\\
 &&\qquad =c T\| \mu\|_{(C_{b,d}(L^{2d}(\mathcal O)))^*}
  = c T \int_{L^{2d}(\mathcal O))}(1+|x|_{L^{2d}(\mathcal O)}^d)| \mu|(dx)<\infty. 
\end{eqnarray*}
Then, \eqref{e.RD.5b} is proved.

We now prove  uniqueness of the solution.
By \eqref{e.L.3} follows that the mild solution $X(t,x)$ of problem \eqref{e.RD.6} can be extended to a process  $(X(t,x))_{t\geq0, x\in H}$ with values in $H$ and adapted to the filtration $(\mathcal F_t)_{t\geq0}$.
In the literature, the process $X(t,x)$   is called a {\em generalized solution} of equation \eqref{e.RD.6} (see \cite{DP04}).
Hence, we can extend the transition semigroup \eqref{e.RD.7} to a semigroup in $C_b(H)$, still denoted by $(P_t)_{t\geq0}$, by setting  
\[
  P_t\varphi(x)=\EE\left[\varphi(X(t,x))\right]\quad t\geq0,\, x\in H,\, \varphi\in C_b(H).
\]
Clearly, $\|P_t\|_{\mathcal L(C_b(H))}\leq 1 $.
In addiction,  
the representation 
\[
 P_t\varphi(x)=\int_H\varphi(y)\pi_t'(x,dy)
\]
holds for any $\varphi\in C_b(H)$, where $ \pi_t'(x,\cdot)$ is the probability measure on $H$ defined by   
$\pi_t'(x,\Gamma)=\PP(X(t,x)\in \Gamma),\, \Gamma \in\mathcal B(H)$.
It  is clear that  $\pi_t'(x,\Gamma)= \pi_t (x,\Gamma)$ when $\Gamma\in \mathcal B(L^{2d}(\mathcal O))$.
We define the infinitesimal generator $K:D(K,C_b(H))\to C_b(H)$ of the semigroup $(P_t)_{t\geq0}$ in the space $C_b(H)$ as in \eqref{e.L.0a}. 
By arguing as in Lemma \ref{l.L.5.3}, the semigroup $(P_t)_{t\geq0}$ in $C_b(H)$ is a stochastically continuous Markov semigroup, in the sense of \cite{Manca07}.
So, we can apply Theorem \ref{t.L.2.2} and then for any $\mu\in \mathcal M(H)$ there exists a unique family of measures $\{\mu_t\}_{t\geq0}\subset \mathcal M(H)$ such that 
\begin{equation}  \label{e.RD.58a}
    \int_0^T|\mu_t|(H)dt<\infty,\quad \forall T>0 
\end{equation}
and \eqref{e.RD.8} hold  for any $t\geq0,\,\varphi\in D(K,C_b(H))$.

Now take $\mu=0$, and assume that $\{\mu_t\}_{t\geq0}\subset \mathcal M_d(L^{2d}(\mathcal O) )$ fulfils \eqref{e.RD.2.5a}, \eqref{e.RD.5b}. 
Since  $\{\mu_t\}_{t\geq0}\subset \mathcal M (H) $, we want to show that $\{\mu_t\}_{t\geq0}$ fulfils also \eqref{e.RD.58a} and \eqref{e.RD.8} for any $t\geq0,\,\varphi\in D(K,C_b(H))$.
Taking in mind that for this equation the solution is unique, this will imply $\mu_t=0$ (as measure in $H$ and consequently as measure in $L^{2d}(\mathcal O) $) for any $t\geq0$.

Clearly, \eqref{e.RD.58a} follows by \eqref{e.RD.2.5a}.
It is also possible to prove, by a standard argument, that $ D(K,C_b(H))\subset D(K)$ and
$ D(K,C_b(H))=\{ \varphi \in D(K): K\varphi\in C_b(H)\}$.
Then, for any $\varphi\in D(K,C_b(H))$, we have $\varphi\in D(K)$ and hence \eqref{e.RD.8}  holds for any $\varphi\in D(K,C_b(H))$.
This concludes the proof.
\qed

%
%
%
%
\subsection{Proof of Theorem \ref{t.RD.2}} \label{s.RD.6.5}
The proof is splitted into two lemmata.
\begin{Lemma} \label{l.RD.6.1}
Let $(K,D(K))$ be the infinitesimal generator \eqref{e.RD.0}.
We have $D(L)\cap C_b^1(H)\subset D(K)\cap C_b^1(H)$ and $K\varphi(x)=L\varphi(x)+\langle D\varphi(x),F(x)\rangle$ for any $\varphi \in D(L)\cap C_b^1(H)$, $x\in L^{2d}(\mathcal O)$.
Moreover,   $(K,D(K))$ is an extension of $K_0$, and for any $\varphi\in \mathcal E_A(H)$ we have $\varphi\in D(K)$ and $K \varphi=K_0\varphi$.
\end{Lemma}
\begin{proof}
Take $\varphi\in \mathcal E_A(H)$.
We   recall that $\mathcal E_A(H)\subset   C_b^1(H)\cap D(L)$, where $(L,D(L))$ was introduced in section \ref{s.RD.OU}. 
We also stress that  since $L^{2d}(\mathcal O)\subset H$, then  $ D(L)\subset C_{b,1}(H)\subset C_{b,d}(L^{2d}(\mathcal O))$ with continuous embedding.
This allow us to proceed as in Theorem \ref{t.L.4.1} to find
\[
   R_t\varphi(x)-\varphi(x)  = P_t\varphi(x)-\varphi(x)
\]
\[
   -\EE\left[\int_0^1\left\langle D\varphi(\xi Z(t,x)+(1-\xi)X(t,x)),\int_0^te^{(t-s)A}F(X(s,x))ds\right\rangle d\xi \right],
\]
for any $x\in L^{2d}(\mathcal O)$.
Hence, by taking into account that $\varphi\in D(L)$, it follows easily that for any $x\in L^{2d}(\mathcal O)$
$$
  \lim_{t\to 0^+} \frac{P_t\varphi(x)-\varphi(x)}{t} = L\varphi(x)+\langle D\varphi(x),F(x)\rangle.
$$
Since there exists $c>0$ such that $|F(x)|\leq c|x|_{L^{2d}(\mathcal O)}^d$, $x\in L^{2d}(\mathcal O)$, it follows
\[
   \sup_{t\in(0,1]}\bigg\|  \frac{P_t\varphi-\varphi}{t}\bigg\|_{C_{b,d}(L^{2d}(\mathcal O))} 
\]
\[
 \leq \sup_{t\in(0,1)}  \bigg\|  \frac{R_t\varphi-\varphi}{t}\bigg\|_{0,1} +\\
\sup_{x\in H} \|D\varphi(x)\|_{\mathcal L(H)}  \sup_{x\in L^{2d}(\mathcal O) }\frac{|F(x)|}{1+|x|_{L^{2d}(\mathcal O)}^d}   <\infty,
\]
that implies $\varphi \in D(K\sostituzioneA)$. 
This proves the first statement.
The fact that    $(K,D(K))$ is an extension of $K_0$ follows 
by Proposition \ref{p.L.4.3}.
\end{proof}
\begin{Lemma} \label{l.RD.6.8}
The set $\mathcal E_A(H)$ is a $\pi$-core for $(K,D(K))$, and for  any $\varphi \in D( {K}\sostituzioneA)$ there exists $m\in \Nset$   and an  $m$-indexed sequence $(\varphi_{n_1,\ldots,n_m})\subset \mathcal E_A(H)$ such that
\begin{eqnarray} \label{e.RD.57}
 \lim_{n_1\to\infty} \cdots \lim_{n_m\to\infty}  \frac{\varphi_{n_1,\ldots,n_m}}{1+|\cdot|_{L^{2d}(\mathcal O)}^d}
  &\stackrel{\pi}{=}& \frac{\varphi }{1+|\cdot|_{L^{2d}(\mathcal O)}^d}, \\
 \lim_{n_1\to\infty} \cdots \lim_{n_m\to\infty}  \frac{K_0\varphi_{n_1,\ldots,n_m}}{1+|\cdot|_{L^{2d}(\mathcal O)}^d}
  &\stackrel{\pi}{=}& \frac{ {K}\varphi }{1+|\cdot|_{L^{2d}(\mathcal O)}^d}. \label{e.RD.58}
\end{eqnarray} 
\end{Lemma}
\begin{proof}
Take $\varphi\in D( {K})$.
We shall construct the claimed sequence in four steps.\\
{\bf Step 1}. 
Fix $\omega>\omega_0, 2(\lambda-\pi^2)$ and set $f=\omega \varphi- {K}\varphi$.
Then we have $\varphi=R(\omega , {K})f$.
We   approximate $f$ as follows: for any $n_1\in \Nset$ we set
\[
   f_{n_1}(x)=\frac{n_1f(e^{\frac{1}{n_1}A}x)}{n_1+|e^{\frac{1}{n_1}A}x) |_{L^{2d}(\mathcal O)}^d},\quad x\in H
\]
By the well known properties of the heat semigroup, we have $e^{\frac{1}{n_1}A}x\in L^{2d}(\mathcal O)$, for any $x\in H$.
Hence, $f_{n_1}\in C_b(H)$ and 
\[
   \lim_{n_1\to\infty}\frac{f_{n_1}}{1+|\cdot|_{L^{2d}(\mathcal O)}^d}\stackrel{\pi}{=}\frac{f }{1+|\cdot|_{L^{2d}(\mathcal O)}^d}.
\]
By Proposition \ref{p.RD.3.1} we have
\[
  \lim_{n_1\to\infty}\frac{P_tf_{n_1}}{1+|\cdot|_{L^{2d}(\mathcal O)}^d}\stackrel{\pi}{=}\frac{P_tf }{1+|\cdot|_{L^{2d}(\mathcal O)}^d}
\]
for any $t\geq0$.
Since we have $\|P_t\|_{\mathcal L(C_{b,d}( L^{2d}(\mathcal O)))}\leq c_0e^{\omega_0 t} $, $\forall t\geq0$ (cf (i) of Proposition \ref{p.RD.3.1}) and $\omega>\omega_0$, it follows
\[
  \lim_{n_1\to\infty}\frac{R(\omega,K)f_{n_1}}{1+|\cdot|_{L^{2d}(\mathcal O)}^d}\stackrel{\pi}{=}
   \frac{R(\omega,K)f }{1+|\cdot|_{L^{2d}(\mathcal O)}^d}.
\]
Setting $\varphi_{n_1}=R(\omega , {K})f_{n_1}$, by the above argument we have
\begin{equation} \label{e.RD.a1}
  \lim_{n_1\to\infty}\frac{\varphi_{n_1}}{1+|\cdot|_{L^{2d}(\mathcal O)}^d}
    \stackrel{\pi}{=}\frac{\varphi }{1+|\cdot|_{L^{2d}(\mathcal O)}^d},\quad
   \lim_{n_1\to\infty}\frac{K\varphi_{n_1}}{1+|\cdot|_{L^{2d}(\mathcal O)}^d}
    \stackrel{\pi}{=}\frac{K\varphi }{1+|\cdot|_{L^{2d}(\mathcal O)}^d}.
\end{equation}
{\bf Step 2}.
For any $n_1\in \Nset$, let us fix a sequence $(f_{n_1,n_2} )_{n_2\in\Nset}\subset C_b^1(H) $ such that
\[
   \lim_{n_2\to\infty} f_{n_1,n_2} \stackrel{\pi}{=} f_{n_1}.
\]
Now set 
$\varphi_{n_1,n_2}=R(\omega ,K)f_{n_1,n_2}$.
By arguing as in step 1 we have
\begin{equation} \label{e.RD.a2}
   \lim_{n_2\to\infty}  \varphi_{n_1,n_2} \stackrel{\pi}{=} \varphi_{n_1},\quad \lim_{n_2\to\infty}  K\varphi_{n_1,n_2} \stackrel{\pi}{=} K\varphi_{n_1}.
\end{equation}
{\bf Step 3}.
We now consider the approximation of $K$. 
We denote by $(K_{n_3},D(K_{n_3}))$ the infinitesimal generator of the transition semigroup associated to the mild solution of problem  \eqref{e.RD.4.16} in the space $C_{b,1}(H)$. 
For any $n_1,n_2,n_3\in \Nset$ set 
\[
  \varphi_{n_1,n_2,n_3}= \int_0^\infty e^{-\omega t}P_t^{n_3}f_{n_1,n_2}dt.
\]
Note that in the right-hand side we have not the {\em resolvent} operator of $K_{n_3}$ in $C_{b,1}(H)$ (cf Proposition \ref{p.L.2.9}, \ref{p.RD.6.7}).
For any $n_1,n_2,n_3\in \Nset$ the function $\varphi_{n_1,n_2,n_3}$ is bounded, since
\[
  \left|\int_0^\infty e^{-\omega t}P_t^{n_3}f_{n_1,n_2}dt\right|\leq \|f\|_0 \int_0^\infty e^{-\omega t}dt<\infty.
\]
The fact that $\varphi_{n_1,n_2,n_3}\in C_b(H)$ follows by standard computations.
By (v) of Proposition \ref{p.L.2.1} and by (i) of Proposition \ref{p.RD.6.7} it follows
\begin{equation} \label{e.RD.a3}
   \lim_{n_3\to\infty}\frac{\varphi_{n_1,n_2,n_3}}{1+|\cdot|_{L^{2d}(\mathcal O)}^d}\stackrel{\pi}{=}
 \frac{\varphi_{n_1,n_2}}{1+|\cdot|_{L^{2d}(\mathcal O)}^d},
\end{equation}
It is also stardard to show that $\varphi_{n_1,n_2,n_3} \in D(K_{n_3})$ 
   and $K_{n_3}\varphi_{n_1,n_2,n_3}=\omega \varphi_{n_1,n_2,n_3}-f_{n_1,n_2}$.
Hence, by \eqref{e.RD.a3} we obtain
\begin{equation} \label{e.RD.a3aa}
 \lim_{n_3\to\infty}\frac{K_{n_3}\varphi_{n_1,n_2,n_3}}{1+|\cdot|_{L^{2d}(\mathcal O)}^d}\stackrel{\pi}{=}
 \frac{K\varphi_{n_1,n_2}}{1+|\cdot|_{L^{2d}(\mathcal O)}^d}
\end{equation}
 By Proposition \ref{p.RD.9}   it follows that $ \varphi_{n_1,n_2,n_3}\in C_b^1(H)$ and
\begin{multline}\label{e.RD.53}
   |D\varphi_{n_1,n_2,n_3}(x)|
  =\left| \int_0^\infty e^{-\omega t}D P_t^{n_3}f_{n_1,n_2}(x)dt \right| \\
  \leq \int_0^\infty e^{-(\omega-2\lambda+2\pi^2) t}dt \sup_{x\in H} |Df_{n_1,n_2}(x)|
 \leq    \frac{ \sup_{x\in H} |Df_{n_1,n_2}(x)|}{\omega-2(\lambda-\pi^2)}.
\end{multline}
Hence $\varphi_{n_1,n_2,n_3}\in  D(K_{n_3})\cap C_b^1(H) $, and by Proposition \ref{p.RD.6.8} it follows that
$K_{n_3}\varphi_{n_1,n_2,n_3}=L\varphi_{n_1,n_2,n_3}+\langle D\varphi_{n_1,n_2,n_3},F_{n_3}\rangle$.
Hence, by Lemma \ref{l.RD.6.1} we have, for any $x\in L^{2d}(\mathcal O)$
\begin{eqnarray} 
  K\varphi_{n_1,n_2,n_3}(x)&=&L\varphi_{n_1,n_2,n_3}(x) +\langle D\varphi_{n_1,n_2,n_3}(x),F(x) \rangle\notag \\
  &=&K_{n_3}\varphi_{n_1,n_2,n_3}(x)+\langle D\varphi_{n_1,n_2,n_3}(x),F(x)-F_{n_3}(x) \rangle.\label{e.RD.a3a}
\end{eqnarray}
We recall that $|F_{n_3}(x)|\leq |F(x)|\leq c|x|_{L^{2d}(\mathcal O)}^d$, for any $ n_3\in\Nset $, $x\in L^{2d}(\mathcal O)$ and for some $c>0$. 
In addiction, $|F_{n_3}(x)- F(x)|\to 0$ as $n_3\to \infty$, for any $x\in L^{2d}(\mathcal O)$.
Consequently, by \eqref{e.RD.53} it follows 
\begin{equation}\label{e.RD.a3b}
  \lim_{n_3\to\infty}\frac{\langle D\varphi_{n_1,n_2,n_3},F-F_{n_3}\rangle }{1+|\cdot|_{L^{2d}(\mathcal O)}^d}\stackrel{\pi}{=}0.
\end{equation}
{\bf Step 4}.
By Propositon \ref{p.L.4.3}
for any $n_1,n_2,n_3\in \Nset$ 
there exists a sequence\footnote{we assume that it has one index} $(\varphi_{n_1,n_2,n_3,n_4})$ $ \subset \mathcal E_A(H)$ such that
\begin{equation} \label{e.RD.a4}
   \lim_{n_4\to\infty}\varphi_{n_1,n_2,n_3,n_4}\stackrel{\pi}{=}\varphi_{n_1,n_2,n_3},
\end{equation}
\begin{eqnarray}\label{e.RD.a5}
   \lim_{n_4\to\infty} \frac{ \frac12\textrm{Tr}\big[BB^*D^2\varphi_{n_1,n_2,n_3,n_4}\big]+\langle x,A^*D\varphi_{n_1,n_2,n_3,n_4}\rangle}{1+|\cdot|}
 \stackrel{\pi}{=}
 \frac{L\varphi_{n_1,n_2,n_3}}{1+|\cdot|}
\end{eqnarray}
and for any $h\in H$ 
\[
 \lim_{n_4\to\infty}\langle D\varphi_{n_1,n_2,n_3,n_4},h\rangle \stackrel{\pi}{=}\langle D\varphi_{n_1,n_2,n_3},h\rangle.
\]
This, together with the above approximation, implies that for any $n_1,n_2,n_3\in\Nset$ we have
\begin{equation} \label{e.RD.a6}
  \lim_{n_4\to\infty}    \frac{ \langle D\varphi_{n_1,n_2,n_3,n_4},F-F_{ n_3}\rangle}{1+|\cdot|_{L^{2d}(\mathcal O)}^d}
   \stackrel{\pi}{=}\frac{\langle D\varphi_{n_1,n_2,n_3},F-F_{ n_3}\rangle }{1+|\cdot|_{L^{2d}(\mathcal O)}^d}.
\end{equation}
{\bf Step 5}.
By \eqref{e.RD.a1}, \eqref{e.RD.a2}, \eqref{e.RD.a3}, \eqref{e.RD.a4} we have
\[
 \lim_{n_1\to\infty}\lim_{n_2\to\infty}\lim_{n_3\to\infty}\lim_{n_4\to\infty} \frac{\varphi_{n_1,n_2,n_3,n_4}}{1+|\cdot|_{L^{2d}(\mathcal O)}^d}
    \stackrel{\pi}{=}\frac{\varphi }{1+|\cdot|_{L^{2d}(\mathcal O)}^d},
\]
and consequently \eqref{e.RD.57} follows.
We now check  
\[
 \lim_{n_1\to\infty}\lim_{n_2\to\infty}\lim_{n_3\to\infty}\lim_{n_4\to\infty} \frac{K_0\varphi_{n_1,n_2,n_3,n_4}}{1+|\cdot|_{L^{2d}(\mathcal O)}^d}
    \stackrel{\pi}{=}\frac{K\varphi }{1+|\cdot|_{L^{2d}(\mathcal O)}^d}.
\]
This will prove \eqref{e.RD.58}.
By Lemma \ref{l.RD.6.1}, for any $n_1,n_2,n_3,n_4\in \Nset $ we have $ K\varphi_{n_1,n_2,n_3,n_4}=K_0\varphi_{n_1,n_2,n_3,n_4}$. 
Moreover, by Theorem \ref{t.L.2} we have $\varphi_{n_1,n_2,n_3,n_4}\in D(K_3)$ and by \eqref{e.RD.a3a}  
\[
 K_0\varphi_{n_1,n_2,n_3,n_4}(x)= K_{n_3}\varphi_{n_1,n_2,n_3,n_4}(x)+\langle D\varphi_{n_1,n_2,n_3,n_4}(x)  ,F(x)-F_{n_3}(x)  \rangle,
\]
for any $n_1,n_2,n_3,n_4\in \Nset $, $x\in L^{2d}(\mathcal O)$.
By \eqref{e.RD.a3a}, \eqref{e.RD.a5}, \eqref{e.RD.a6} it holds
\begin{eqnarray*}
   \lim_{n_4\to\infty} \frac{K_0\varphi_{n_1,n_2,n_3,n_4}}{1+|\cdot|_{L^{2d}(\mathcal O)}^d} &\stackrel{\pi}{=}& 
 \frac{K_{n_3}\varphi_{n_1,n_2,n_3 }+\langle D\varphi_{n_1,n_2,n_3 },F-F_{ n_3}  \rangle}{1+|\cdot|_{L^{2d}(\mathcal O)}^d}.
\end{eqnarray*}
By  \eqref{e.RD.a3aa}, \eqref{e.RD.a3b} it holds
\[
  \lim_{n_3\to\infty} \frac{K_{n_3}\varphi_{n_1,n_2,n_3 }+\langle D\varphi_{n_1,n_2,n_3 },F-F_{ n_3}  \rangle}{1+|\cdot|_{L^{2d}(\mathcal O)}^d}\stackrel{\pi}{=} \frac{K\varphi_{n_1,n_2}}{1+|\cdot|_{L^{2d}(\mathcal O)}^d}.
\]
By  \eqref{e.RD.a1},  \eqref{e.RD.a2} it holds
\[
 \lim_{n_1\to\infty} \lim_{n_2\to\infty}  \frac{K\varphi_{n_1,n_2}}{1+|\cdot|_{L^{2d}(\mathcal O)}^d} 
   \stackrel{\pi}{=} \frac{K\varphi }{1+|\cdot|_{L^{2d}(\mathcal O)}^d}. \qedhere
\]
\end{proof}

\subsection{Proof of Theorem \ref{t.RD.1.4}} \label{s.RD.6.6}
 
Take $\mu\in \mathcal M_d(L^{2d}(\mathcal O) )$. 
The fact that $\{P_t^*\mu\}_{t\geq0}$ fulfils \eqref{e.RD.5b} and 
\eqref{e.RD.2.5a} follows by Theorems \ref{t.RD.1}, \ref{t.RD.2} and by the fact that $KP_t\varphi=P_tK\varphi=P_tK_0\varphi$, for any $\varphi\in \mathcal E_A(H)$ (cf Proposition \ref{p.RD.6.0} and Lemma \ref{l.RD.6.1}).
Hence, existence of a solution is proved.
Now we show uniqueness.  
Assume that $\{\mu_t\}_{t\geq0}\subset \mathcal M_d(L^{2d}(\mathcal O) )$ fulfils \eqref{e.RD.5b} and 
\eqref{e.RD.2.5a}.
By Theorem \ref{t.RD.2} for any $\varphi\in D(K )$  there exist $m\in\Nset$ and an $m$-indexed  sequence $(\varphi_{n_1,\ldots,n_m})_{n_1\in\Nset,\ldots,n_m\in\Nset}\subset \mathcal E_A(H)$ such that
\eqref{e.RD.48}, \eqref{e.RD.49} hold.   
This, togheter with \eqref{e.RD.5b}, implies that $\{ \mu_t\}_{t\geq0} $ fulfils \eqref{e.RD.8} for any $t\geq0,\,\varphi\in D(K)$ (here we can use the same argument used to prove Theorem \ref{t.L.1.4}).
Since the solution of \eqref{e.RD.5b}, \eqref{e.RD.8} is unique and it is given by $\{P_t^*\mu\}_{t\geq0}$,
it follows $\int_H\varphi(x)P_t^*\mu(dx)=\int_H\varphi(x) \mu_t(dx) $, for any $\varphi\in \mathcal E_A(H)$.
Hence, since $\mathcal E_A(H)$ is $\pi$-dense in $C_b(H)$, it follows $\int_H\varphi(x)P_t^*\mu(dx)=\int_H\varphi(x) \mu_t(dx) $, for any $\varphi\in C_b(H)$, that implies $P_t^*\mu=\mu_t$, $\forall t\geq0$.
This concludes the proof.
\qed


\begin{thebibliography}{99}
\bibitem{BDPR04}
{ V. I. Bogachev, G. Da Prato  and M. R{\"o}ckner},
     `Existence of solutions to weak parabolic equations for
              measures',
   {\em Proc. London Math. Soc. (3)},
    {88},
       (2004),
    no. {3},
     {753--774},


\bibitem{BR00}
{ V. I. Bogachev  and M. R{\"o}ckner},
     `A generalization of {K}hasminskii's theorem on the
              existence of invariant measures for locally integrable drifts',
   {\em Teor. Veroyatnost. i Primenen.},
      {45},
       (2000),
   no. {3},
  {417--436},
   translation in
{\em Theory Probab. Appl. 45} (2002), no. 3, 363--378 

\bibitem{BR01}
{ V. I. Bogachev  and M. R{\"o}ckner},
     `Elliptic equations for measures on infinite-dimensional spaces
              and applications',
   {\em Probab. Theory Related Fields},
     {120},
      (2001),
     {no. 4},
     {445--496},

\bibitem {Cerrai}
{ S. Cerrai},   
     `A {H}ille-{Y}osida theorem for weakly continuous semigroups',
  {\em Semigroup Forum},
      49
     (1994),            
  no.   {3},
       {349--367},        

\bibitem{Cerrai01}
{ S. Cerrai},
     `Second order {PDE}'s in finite and infinite dimension',
     {Lecture Notes in Mathematics, 1762},
        {A probabilistic approach},
  {\em Springer-Verlag, Berlin},
      {2001},

\bibitem{DP04}
{G. Da Prato},
   `Kolmogorov equations for stochastic {PDE}s',
   {Advanced Courses in Mathematics. CRM Barcelona},
    {\em Birkh\"auser Verlag, Basel},
   {2004}

\bibitem{DPZ92}
{ G. Da Prato and J. Zabczyk},
  `Stochastic equations in infinite dimensions',
    {Encyclopedia of Mathematics and its Applications},
    {44},
    {\em Cambridge University Press, Cambridge}
    {1992}

\bibitem{DPZ96}
{ G. Da Prato and J. Zabczyk},
   `Ergodicity for infinite-dimensional systems',
   {London Mathematical Society Lecture Note Series},
    {229},
   {\em Cambridge University Press, Cambridge},
    {1996}

\bibitem{Manca07}
{L. Manca},
   `Kolmogorov equations for measures',
     (2007),
     {Preprint}

\bibitem{Priola}
{ E. Priola},
   `On a class of {M}arkov type semigroups in spaces of uniformly
   continuous and bounded functions',
  {\em Studia Math.},
    {136},
    (1999),
   no.  {3},
   {271--295}
%
%
\end{thebibliography}
\end{document}